\documentclass{amsart}

\usepackage[top=2cm, bottom=2.8cm, left=2cm, right=2cm]{geometry}
\usepackage[english]{babel}
\usepackage{amsmath}
\usepackage{amsthm}
\usepackage{amsfonts}
\usepackage{amssymb}
\usepackage{graphicx}
\usepackage{mathrsfs}
\usepackage{a4wide,color,eucal,enumerate,mathrsfs}
\usepackage[normalem]{ulem}
\usepackage{amsmath,amssymb,epsfig,bbm}
\usepackage{pdfsync}
\usepackage{tikz-cd}
\usepackage{ esint }
\usepackage{textcomp}
\numberwithin{equation}{section}
\usepackage[pdfborder={0 0 0}]{hyperref}  
\usepackage{adjustbox}
\usepackage{caption}
\newcommand{\M}{\mathbb{M}}
\newcommand{\N}{\mathbb{N}}
\newcommand{\Q}{\mathbb{Q}}
\newcommand{\R}{\mathbb{R}}
\newcommand{\Z}{\mathbb{Z}}

\newcommand{\cB}{{\ensuremath{\mathcal B}}}

\newcommand{\cD}{{\ensuremath{\mathcal D}}}

\newcommand{\mm}{{\mbox{\boldmath$m$}}}

\newcommand{\DDelta}{{\mbox{\boldmath$\Delta$}}}

\newcommand{\sfd}{{\sf d}}

\newcommand{\sfr}{{\sf r}}

\newcommand{\sfS}{{\sf S}}

\newcommand{\restr}[1]{\lower3pt\hbox{$|_{#1}$}}

\newcommand{\la}{{\big\langle}}                  
\newcommand{\ra}{{\big\rangle}}

\newcommand{\eps}{\varepsilon}  
\newcommand{\nchi}{{\raise.3ex\hbox{$\chi$}}}
\newcommand{\weakto}{\rightharpoonup}

\newtheorem{theorem}{Theorem}[section]

\newtheorem{corollary}[theorem]{Corollary}
\newtheorem{lemma}[theorem]{Lemma}
\newtheorem{proposition}[theorem]{Proposition}
\newtheorem{definition}[theorem]{Definition}

\newtheorem{example}[theorem]{Example}
\newtheorem{remark}[theorem]{Remark}

\newcommand{\Lip}{\mathrm{Lip}}

\newcommand{\diam}{\mathrm{diam}}

\newcommand{\fr}{\hfill$\blacksquare$}   

\newcommand{\RCD}{\mathrm{RCD}}

\renewcommand{\mm}{\mathfrak m}

\newcommand{\Cat}[1]{{\sf CAT}(#1)}

\newcommand{\lims}{\varlimsup}
\newcommand{\limi}{\varliminf}
\renewcommand{\limsup}{\varlimsup}
\renewcommand{\liminf}{\varliminf}
\renewcommand{\d}{{\rm d}}

\newcommand{\Geo}{{\sf Geo}}

\newcommand{\X}{{\rm X}}
\newcommand{\Y}{{\rm Y}}
\newcommand{\T}{{\rm T}}

\newcommand{\G}{{\sf G}}

\newcommand{\subd}[1]{\partial E ({#1})}

\newcommand{\sfks}{{\sf ks}}
\newcommand{\sfKS}{{\sf KS}}
\newcommand{\E}{{\sf E}}

\newcommand{\sfFL}{{\sf FI}}
\newcommand{\msp}{|\dot{y}_t|}
\title{A differential perspective on  Gradient Flows on $\Cat\kappa$-spaces and applications}
\author{Nicola Gigli}
\address{SISSA, Via Bonomea 265, 34136 Trieste}
\email{ngigli@sissa.it}
\author{Francesco Nobili}
\address{SISSA, Via Bonomea 265, 34136 Trieste}
\email{fnobili@sissa.it}

\begin{document}
\maketitle

\begin{abstract}
We review the theory of Gradient Flows in the framework of convex and lower semicontinuous functionals on $\Cat\kappa$-spaces and prove that they can be characterized by the same differential inclusion $y_t'\in-\partial^-\E(y_t)$ one uses in the smooth setting and more precisely that $y_t'$ selects the element of minimal norm in $-\partial^-\E(y_t)$. This generalizes previous results in this direction where the energy was also assumed to be Lipschitz.

We then apply such result to the Korevaar-Schoen energy functional on the space of $L^2$ and \Cat0 valued maps: we define the Laplacian of such $L^2$ map as the element of minimal norm in $-\partial^-\E(u)$, provided it is not empty. The theory of gradient flows ensures that the set of maps admitting a Laplacian is $L^2$-dense. Basic properties of this Laplacian are then studied.

\end{abstract}

\tableofcontents

\bigskip
\section{Introduction}\label{Sect1}
The theory of gradient flows in metric spaces has been initiated by De Giorgi and collaborators \cite{DeGiorgiMarinoTosques80}, \cite{DeGiorgi93} (see also the more recent \cite{AmbrosioGigliSavare08}): a basic feature of the approach is to provide a very general existence theory - at this level uniqueness is typically lost - without neither curvature assumptions on the space nor semiconvexity of the functional. 

In this setting gradient flow trajectories $(x_t)$ of $\E$  (or curves of maximal slopes) are  defined by imposing the maximal rate of dissipation
\[
\frac{\d}{\d t}\E(x_t)=-|\dot x_t|^2=-|\partial^-\E|^2(x_t),\qquad a.e.\ t,
\]
where here $|\dot x_t|$ is the metric speed of the curve (see Theorem \ref{thm:ms}) and $|\partial^-\E|$ is the slope of $\E$  (see \eqref{eq:defsl}). It has been later understood (\cite{AmbrosioGigliSavare08}, \cite{AmbrosioGigliSavare11-2}, \cite{Gigli12}, \cite{OP17}, \cite{MS20}) that if  $\E$ is $\lambda$-convex and the metric space has some form of some Hilbert-like structure at small scales, then an equivalent formulation can be given via the  so-called Evolution Variational Inequality
\begin{equation}
\frac{\d}{\d t}\frac{\sfd^2(x_t,y)}2+\E(x_t)+\frac\lambda2\sfd^2(x_t,y)\leq \E(y)\qquad a.e.\ t
\tag{EVI}
\end{equation}

for any choice of point $y$ on the space. See Theorem \ref{thm:GFdef} for the precise definitions and \cite{MS20} for a thorough study of the EVI condition.

The geometry of the metric space and the convexity properties of the functional under consideration greatly affect the kind of results one can obtain for gradient flows. For the purpose of this manuscript, the works \cite{Mayer98}, \cite{Jost98} are particularly relevant: it is showed that the classical Crandall-Liggett generation theorem can be generalized to the metric setting of \Cat0 spaces to produce a satisfactory theory of gradient flows for semi-convex and lower semicontinuous functionals. 

If the metric space one is working on admits some nicely-behaved tangent spaces/cones, one might hope to give a meaning to the classical defining formula
\[
x_t'\in-\partial^-\E(x_t)\qquad a.e.\ t
\] 
or to its more precise variant
\begin{equation}
\label{eq:gfi}
x_t'^+=\text{ the element of minimal norm in }-\partial^-\E(x_t)\qquad\forall t>0.
\end{equation}
This has been done in \cite{Lyt05}, where previous approaches in \cite{PP95} have been generalized. Here, notably, the basic assumptions on the metric space are of first order in nature (and refer precisely to the structure of tangent cones) and the energy functional is assumed to be semiconvex and locally Lipschitz. While the convexity assumption is very natural when studying gradient flows (all in all, even in the Hilbert setting many fundamental results rely on such hypothesis), asking for Lipschitz continuity is a bit less so: it certainly covers many concrete examples, for instance of functionals built upon distance functions on spaces satisfying some one-sided curvature bound, but from the analytic perspective it may be not satisfying: already the Dirichlet energy as a functional on $L^2$ is not Lipschitz, and the same holds for the Korevaar-Schoen energy we aim to study here.

\bigskip

Our motivation to study this topic comes from the desire of providing a notion of Laplacian for \Cat0-valued Sobolev maps, where here `Sobolev' is intended in the sense of Korevaar-Schoen \cite{KS93} (see also the more recent review of their theory done in \cite{GT20}). Denoting by $\E^{\sfKS}$ the underlying notion of energy and imitating one of the various equivalent definitions for the Laplacian in the classical smooth and linear setting, one is lead to define the Laplacian of $u$ as the element of minimal norm in $-\partial^-\E^\sfKS(u)$. This approach of course carries at least two tasks: to define what $-\partial^-\E$ is and to show that it is not empty for a generic convex and lower semicontinuous functional $\E$. Providing a reasonable definition for $-\partial^-\E$ is not that hard (see Definition \ref{def:md}), but is less obvious to show that this object is not-empty (in particular, minimizing $\E(\cdot)+\frac{\sfd^2(\cdot,x)}{2\tau}$ is of no help here, see the discussion in Remark \ref{re:mm}).  It is here that the theory of gradient flows comes to help:  

\bigskip

\begin{quote}
our main result is that, for semiconvex and lower semicontinuous functions on a $\Cat\kappa$ space, the analogue of \eqref{eq:gfi} holds, see Theorem \ref{thm:rightD}. 
\end{quote}

\bigskip

As a byproduct, we deduce that the domain of $-\partial^-\E$ is dense in the one of $\E$. A result similar to ours has been obtained in \cite{CKK19} under some additional geometric assumptions on the base space, which in some sense tell that there is the opposite of any tangent vector.

As said, we then apply this result to study the Laplacian of $\Cat0$-valued Sobolev maps. Let us remark that in this case the relevant metric space $L^2(\Omega,\Y_{\bar y})$ is that of $L^2$ maps from some open subset $\Omega$ of a metric measure space $\X$ to a pointed \Cat0 space $(\Y,\bar y)$ and the energy functional is the  Korevaar-Schoen energy $\E^\sfKS$: it is well known that $L^2(\Omega,\Y_{\bar y})$ is a \Cat0 space and that  $\E^\sfKS$ is convex and lower semicontinuous, but certainly not Lipschitz, whence the need to generalize Lytchak results to cover also this case.

Once we have a notion for $-\partial^-\E^\sfKS$ we enrich the paper with:
\begin{itemize}
\item[i)]  the actual definition of Laplacian $\Delta u$ of a $\Cat0$-valued map $u$ (Definition \ref{def:laplacian}), which pays  particular attention to the link between the tangent cones in $L^2(\Omega,\Y_{\bar y})$, where $-\partial^-\E^\sfKS$ lives, and the tangent cones in $\Y$, where we think `variations' of $u$ should live, see in particular Propositions \ref{prop:l2gen} and  \ref{prop:link},
\item[ii)] a basic, weak, integration by parts formula, see Proposition \ref{prop:varen}, which is sufficient to show that our approach is compatible with the classical one valid in the smooth category,
\item[iii)] a presentation of  a simple and concrete example (Example \ref{ex:s1}) showing why $\Delta u$ seems to be very much linked to the geometry of $\Y$, but less so to Sobolev calculus on it. 
\end{itemize}
\bigskip


Finally, we point out that this note is part of a larger program aiming at stating and proving the Eells-Sampson-Bochner inequality \cite{ES64} for Sobolev maps from (open subsets of a) $\RCD$ space $\X$ to a $\Cat0$ space $\Y$ (see \cite{DMGSP18,GPS18,GT20} for partial results in this direction): knowing what the Laplacian of a \Cat0-valued map is, is a crucial step for this program.

\bigskip{\bf Acknowledgement.} We thank A.\ Lytchak and M.\ Ba\v{c}\'{a}k for comments on an preliminary version of this manuscript. 

\section{Calculus on ${\sf CAT}(\kappa)$-spaces}\label{Sect2}
\subsection{$\Cat\kappa$-spaces}
Let us briefly recall some useful tools in metric spaces $(\Y,\sfd_{\Y})$.
\begin{definition}[Locally AC curve]
Let $(\Y,\sfd_{\Y})$ be a metric space and let $I \subset \R$ be an interval. A curve $I \ni t \mapsto \gamma_t \in \Y$ is \emph{absolutely continuous} if there exists a function $g \colon I\mapsto \R^+$ in $L^1(I)$ s.t.
\begin{equation}
    \sfd_{\Y}(\gamma_t,\gamma_s)\le \int_s^tg(r)\, \d r \qquad\forall s\le t \emph{ in } I. \label{eq:AC}
\end{equation}
Moreover, $\gamma$ is said to be locally absolutely continuous if every point admits a neighbourhood where it is absolutely continuous.
\end{definition}
Next, we state the existence of the metric counterpart of `modulus of velocity' of a curve.
\begin{theorem}[Metric speed]\label{thm:ms}
Let $(\Y,\sfd_{\Y})$ be a  metric space and let $I \subset \R$ be an interval. Then, for every AC curve $ I\ni t \mapsto \gamma_t \in \Y$, there exists the limit
$$ \lim_{h \downarrow 0} \frac{\sfd_{\Y}(\gamma_{t+h},\gamma_t)}{h} \qquad \text{a.e. } t  \in I,$$
which we denote by $\vert \dot{\gamma}_t\vert $ and call \emph{metric speed}. Moreover, it is the least, in the a.e. sense, function $L^1(I)$ that can be taken in (\ref{eq:AC}).
\end{theorem}
See, for the proof, \cite[Theorem 1.1.2]{AmbrosioGigliSavare08}. A curve $[0,1] \ni t\mapsto \gamma_t \in \Y$ is a minimizing constant speed geodesic (or  simply a geodesic) if $\sfd_{\Y}(\gamma_t,\gamma_s) =\vert t-s\vert \sfd_{\Y}(\gamma_0,\gamma_1)$, for every $t,s\in [0,1]$. We say that $\Y$ is a geodesic metric space provided for any couple of points, there exists a constant speed geodesic joining them. Whenever the geodesic connecting $y$ to $z$  is unique, we shall denote it by $\G_y^z$.

For $\kappa\in\R$, we call $\M_\kappa$, the \emph{model space} of curvature $\kappa$, i.e. the simply connected, complete 2-dimensional manifold with constant curvature $\kappa$, and $\sfd_\kappa$ the distance induced by the metric tensor. 
This restricts $(\M_\kappa,\sfd_\kappa)$ to only three possibilities: the hyperbolic space $\mathbb H^2_\kappa$ of constant sectional curvature $\kappa$, if $\kappa<0$, the plane $\R^2$ with usual euclidean metric, if $\kappa=0$, and the sphere $\mathbb{S}^2_\kappa$ of constant sectional curvature $\kappa$, if $\kappa>0$. Also, set $D_\kappa:=\diam(\M_\kappa)$, i.e.
\begin{align*}
D_\kappa=\left\{
\begin{array}{ll}
\infty&\quad\text{ is }\kappa\le 0,\\
\frac{\pi}{\sqrt\kappa}&\quad\text{ if }\kappa>0.
\end{array}
\right.
\end{align*}
We refer to \cite[Chapter I.2]{BH99} for a detailed study of the model spaces $\M_\kappa$.

\bigskip

In order to speak of $\kappa$-upper bound of the sectional curvature in a geodesic metric space $(\Y,\sfd_\Y)$, we shall enforce a metric comparison property to geodesic triangles of $\Y$, the intuition being that they are `thinner' than in $\M_\kappa$. To define them we start by recalling that if $a,b,c\in\Y$ is a triple of points satisfying $\sfd_\Y(a,b)+\sfd_\Y(b,c)+\sfd_\Y(c,a)<2D_\kappa$, then there are points, unique up to isometries of the ambient space and called \emph{comparison points}, $\bar a,\bar b,\bar c\in\M_\kappa$ such that
\[
\sfd_\kappa(\bar a,\bar b)=\sfd_\Y(a,b),\qquad\qquad\sfd_\kappa(\bar b,\bar c)=\sfd_\Y(b,c),\qquad\qquad\sfd_\kappa(\bar c,\bar a)=\sfd_\Y(c,a).
\]
In the case where $\Y$ is geodesics (and this will be always assumed), we refer to $\triangle(a,b,c)$ as the geodesic triangle in $\Y$ consisting in three points $a,b,c$, the \emph{vertices}, and a choice of three corresponding geodesics, the \emph{edges}, linking pairwise the points. By $\triangle^{\kappa}(\bar{a},\bar{b},\bar{c})$ we denote the so built geodesic triangle in $\M_\kappa$, which from now on we call comparison triangle. A point $d\in\Y$ is said to be intermediate between $b,c\in\Y$ provided $\sfd_\Y(b,d)+\sfd_\Y(d,c)=\sfd_\Y(b,c)$ (this means that $d$ lies on a geodesic joining $b$ and $c$). The \emph{comparison point of $d$} is the (unique, once we fix the comparison triangle) point $\bar d\in\M_\kappa$, such that
\[
\sfd_\kappa(\bar d,\bar b)=\sfd_\Y(d,b),\qquad\qquad\sfd_\kappa(\bar d,\bar c)=\sfd_\Y(d,c).
\]
\begin{definition}[$\Cat\kappa$-spaces]\label{cat}
A  metric space $(\Y,\sfd_\Y)$ is called a $\Cat\kappa $-space if it is complete, geodesic and satisfies the following triangle comparison principle: for any $a,b,c\in \Y$ satisfying $\sfd_\Y(a,b)+\sfd_\Y(b,c)+\sfd_\Y(c,a)<2D_\kappa$ and any intermediate point $d$ between $b,c$, denoting by  $\triangle^\kappa(\bar a,\bar b,\bar c)$  the  comparison triangle  and by $\bar d\in\M_\kappa$ the corresponding  comparison point (as said, $\bar a,\bar b,\bar c,\bar d$ are unique up isometries of $\M_\kappa$),  it holds
\begin{equation}
\label{eq:defcat}
\sfd_\Y(a,d)\le \sfd_\kappa(\bar  a,\bar d).
\end{equation}
A  metric space $(\Y,\sfd_\Y)$ is said to be locally $\Cat\kappa$ if it is complete, geodesic and every point in $\Y$ has a neighbourhood which is a $\Cat\kappa$-space with the inherited metric. 
\end{definition}
Notice that balls of radius  $<D_\kappa/2$ in the model space $\M_\kappa$ are convex, i.e. meaning that geodesics with endpoint them lies entirely inside. Hence the  comparison property \eqref{eq:defcat} grants that the same is true on $\Cat\kappa$-spaces (see \cite[Proposition II.1.4.(3)]{BH99} for the rigorous proof of this fact). It is then easy to see that, for the same reasons, $(\Y,\sfd_\Y)$ is locally $\Cat\kappa$  provided every point has a neighbourhood $U$ where the comparison inequality (\ref{eq:defcat}) holds for every triple of points $a,b,c\in U$, where the geodesics connecting the points (and thus the intermediate points) are allowed to exit the neighbourhood $U$.

Let us fix the following notation: if $(\Y,\sfd_\Y)$ is a local $\Cat\kappa$-space, for every $y\in\Y$ we set
\[
\sfr_y:=\sup\big\{r\leq D_\kappa/2\ :\ \bar B_r(y)\ \text{is a $\Cat\kappa$-space}\big\}.
\]
Notice that in particular $B_{\sfr_y}(y)$ is a $\Cat\kappa$-space. The definition trivially grants that $\sfr_y\geq\sfr_z-\sfd(y,z)$ and thus in particular $y\mapsto \sfr_y$ is continuous.

Finally, we remark the important fact which will be exploited in the sequel
\begin{quote}
\label{eq:geocontdep} On $\Cat\kappa$-spaces, geodesics with endpoint at distance $<D_\kappa$ are unique (up to reparametrization) and vary continuously with respect to the endpoints.
\end{quote}
For a quantitative version of this fact, see \cite[Lemma 2.2]{DMGSP18}.
Finally, it will be important to examine the case of global \Cat0-spaces, as they naturally arise as tangent structures of $\Cat\kappa$-spaces (see Theorem \ref{thm:tancat} below) and also because we are going to examine \Cat0-valued maps in Section \ref{Sect4}. Since $\M_0$ is the euclidean plane $\R^2$ equipped with the euclidean norm, for $\Y$ \Cat0 and $a,b,c \in \Y$ as in Definition \ref{cat}, the defining inequality \eqref{eq:defcat} reads
\[
    \sfd_\Y(\gamma_t,a) \le \Vert (1-t)\bar{b}+t\bar{c}-\bar{a}\Vert, 
\]
for every $t\in [0,1]$, where $\gamma_t$ is the constant speed geodesic connecting $b$ to $c$ and $\bar{a},\bar{b},\bar{c} \in \R^2$ are comparison points. By squaring and expanding the right hand side, we easily obtain the condition
\begin{equation}
    \sfd_\Y^2(\gamma_t, a) \le (1-t)\sfd_\Y^2(\gamma_0, a)+ t\sfd_\Y^2(\gamma_1,a) -t(1-t)\sfd_\Y^2(\gamma_0,\gamma_1), \label{eq:cat0def}
\end{equation} 
for every $t \in [0,1]$. Inequality \eqref{eq:cat0def} (which can be equivalently used to define \Cat0-spaces) is to be understood as a synthetic deficit of the curvature of $\Y$, with respect to the euclidean plane $\R^2$ (where it holds with equality). In other  words, it quantifies how much the triangle $\triangle(a,b,c)$ is `thin' compared to $\triangle^0(\bar{a},\bar{b},\bar{c})$ in the euclidean plane. The advantage of \eqref{eq:cat0def} is to be more practical to work with in convex analysis and optimization.

\subsection{Tangent cone} We recall here the notion of tangent cone on a $\Cat\kappa$-space, referring to the above-mentioned bibliography for a much more complete discussion.
 
 We define the tangent cone of a $\Cat\kappa$-space by means of geodesics. Let us start with some considerations valid in a general geodesic space  $\Y$: as we shall see, the construction is valid on this generality, but it will benefit from the $\Cat\kappa$ condition making a suitable calculus possible. 
 
Let $ y \in \Y$, and denote by $\Geo_y\Y$ the space of (constant speed) geodesics emanating from $y$ and defined on some right neighbourhood of $0$. We endow this space with the pseudo-distance $\sfd_y $ defined as following:
\begin{equation}
 \sfd_y (\gamma,\eta ) := \limsup_{t\downarrow 0}\frac{\sfd_{\Y}(\gamma_t,
\eta_t)}{t} \qquad \forall \gamma, \eta \in \Geo_y\Y. \label{eq:dy}
\end{equation}
It is easy to see that $\sfd_y$ naturally induces an equivalence relation on $\Geo_y\Y$, by simply imposing $\gamma \sim \eta$ if $\sfd_y(\gamma,\eta)=0.$ By construction, $\sfd_y$ passes to the quotient $\Geo_y\Y/\sim$ and with (a common) abuse of notation, we still denote $\sfd_y$ the distance on the quotient space. The equivalence class of the geodesic $\gamma$ under this relation will be denoted $\gamma'_0$. In particular this applies to the geodesics $\G_y^z$ defined on $[0,1]$, whose corresponding element in $\Geo_y\Y/\sim$ will be denoted by $(\G_y^z)'_0$. 

\begin{definition}[Tangent cone]
Let $\Y$ be a geodesic space and $y \in \Y$. The \emph{tangent cone} $(\T_y\Y,\sfd_y)$, is the completion of  $(\Geo_y\Y/\sim,\sfd_y)$. Moreover, we call $0_y \in \T_y\Y$, the equivalence class of the steady geodesic at $y$.
\end{definition}
A direct consequence of the local $\Cat\kappa$ condition is that, for every $y \in \Y, \gamma,\eta \in \Geo_y\Y$, the limsup in \eqref{eq:dy} is actually a limit. It will be also useful to notice that
 \begin{equation}
 \text{if $\Y$ is \Cat0, } t\mapsto \frac{\sfd_\Y(\gamma_t,\eta_t)}{t} \text{ is non-decreasing}\qquad \forall \gamma,\eta \in \Geo_y\Y, \label{eq:mondis}
\end{equation}
a property which is directly implied by \eqref{eq:cat0def}. A well known (see e.g.\  \cite[Theorem II-3.19]{BH99}) and useful fact is that  tangent cones at local $\Cat\kappa$ spaces are $\Cat0$ spaces:
\begin{theorem}\label{thm:tancat}
Let $\Y$ be locally $ \Cat\kappa$. Then, for every $y \in \Y$, the tangent cone $(\T_y\Y,\sfd_y)$ is a \Cat0-space.
\end{theorem}
We now build a calculus on the tangent cone that resembles the one of Hilbert spaces.
\begin{itemize}
\item[\textopenbullet] \emph{Multiplication by a positive scalar}. Let $\lambda\geq 0$. Then the map sending $t\mapsto \gamma_t$ to   $t\mapsto\gamma_{\lambda t}$ is easily seen to pass to the quotient in $\Geo_y\Y/\sim$ and to be $\lambda$-Lipschitz. Hence it can be extended by continuity to a map defined on $\T_y\Y$,  called multiplication by $\lambda$.
\item[\textopenbullet] \emph{Norm}. $|v|_y:=\sfd_y(v,0)$. 
\item[\textopenbullet] \emph{Scalar product}. $\la v,w\ra_y: = \tfrac12\big[|v|_y^2+|w|_y^2-\sfd_y^2(v,w)\big]$.
\item[\textopenbullet] \emph{Sum}. $v\oplus w:=2 m$, where $m$ is the midpoint of $v,w$ (well-defined because $\T_y\Y$ is a $\Cat0$-space).
\end{itemize}
We report from \cite[Theorem 2.9]{DMGSP18} the following fact:
\begin{equation}
\label{eq:densecone}
\text{for $\cD$ dense in $B_{\sfr_y}(y)$ we have that $\{\alpha(\G_y^w)'_0\colon \alpha \in \Q^+, \ w \in \cD\}$ is dense in $\T_y\Y$}.
\end{equation}
Moreover, we recall the following proposition:
\begin{proposition}[Basic calculus on the tangent cone]\label{prop:hilbertine}
Let $\Y$ be locally $ \Cat\kappa$ and  $y \in \Y$.  Then, the four operations defined above are continuous in their variables. The `sum' and the `scalar product' are also symmetric. Moreover:
\begin{subequations}
\begin{align}
\label{eq:norm}
\sfd_y(\lambda v,\lambda w)&=\lambda \sfd_y(v,w),\\
\label{eq:prhom}
\la\lambda v,w\ra_y&= \la v,\lambda w\ra_y=\lambda \la v,w\ra_y,\\
\label{eq:CS}
|\la v,w\ra_y|&\le |v|_y|w|_y,\\
\label{eq:CSeq}
\la v,w\ra_y&= |v|_y|w|_y\quad\text{ if and only if }\quad |w|_yv=|v|_yw,\\
\label{eq:PI}
\sfd_y^2(v,w)+|v\oplus w|_y^2&\le 2(|v|_y^2+|w|_y^2),\\
\label{eq:concav}
\la v_1\oplus v_2,w\ra_y&\geq \la v_1,w\ra_y+\la v_2,w\ra_y
\end{align}
\end{subequations}	
for any $v,v_1,v_2,w\in \T_y\Y$  and $\lambda\geq 0$.
\end{proposition}
\begin{proof} The continuity of `norm', `scalar product' and `multiplication by a scalar' are obvious by definition, the one of `sum' then follows from the continuity of the midpoint of a geodesic as a function of the extremal points.

Points \eqref{eq:norm}, \eqref{eq:prhom}, \eqref{eq:CS}, \eqref{eq:CSeq}, \eqref{eq:PI} are well known and recalled, e.g., in \cite[Proposition 2.11]{DMGSP18}. The concavity property \eqref{eq:concav} is also well known. A way to prove it is to notice that from \eqref{eq:prhom} and letting $m$ be the midpoint of $v_1,v_2$ we get that
\[
\la v_1\oplus v_2,w\ra_y=2\eps^{-1}\la \eps m,w\ra_y=\eps^{-1}\big(\eps^2|m|^2_y+|w|_y^2-\sfd_y^2(\eps m,w)\big)\qquad\forall \eps>0.
\]
From the fact that $\T_y\Y$ is \Cat0 and the fact that $\eps m$ is the midpoint of $\eps v_1,\eps v_2$ (consequence of \eqref{eq:norm}) we get that $\sfd_y^2(\eps m,w)\leq \frac12\sfd_y^2(\eps v_1,w)+\frac12 \sfd_y^2(\eps v_2,w)$ and plugging this in the above we get
\[
\begin{split}
\la v_1\oplus v_2,w\ra_y&\geq\eps^{-1}\big(\tfrac12\big(|w|_y^2-\sfd_y^2(\eps v_1,w)\big)+\tfrac12\big(|w|_y^2-\sfd_y^2(\eps v_2,w)\big)\big)\\
&=\la v_1,w\ra_y+\la v_2,w\ra_y-\tfrac\eps2(|v_1|^2_y+|v_2|^2_y)\qquad\forall \eps>0
\end{split}
\]
and the conclusion follows letting $\eps\downarrow0$.
\end{proof}
It will also be useful to know that
\begin{equation}
\label{eq:sumexpl}
\alpha(\G_y^z)'_0\oplus\beta(\G_y^w)'_0=\lim_{t\downarrow0}\frac2\eps(\G_y^{m_t})'_0,
\end{equation}
for $z,w \in B_{\sfr_y}(y)\setminus\{y\}$, where $m_t$ is the midpoint of $(\G_y^z)_{\alpha t}$ and  $(\G_y^z)_{\beta t}$, see for instance \cite[II-Theorem 3.19]{BH99} for the simple proof.

We conclude recalling that on $\Cat\kappa$-spaces not only a notion of metric derivative is in place for absolutely continuous curves, but it is possible to speak about right (or left) derivatives in the following sense, as proved in \cite{Lyt04}:
\begin{proposition}[Right derivatives]\label{prop:rlder}
 Let $\Y$ be locally $ \Cat\kappa$ and $(y_t)$ an absolutely continuous curve. Then, for a.e.\ $t$, the tangent vectors $\frac1h(\G_{y_t}^{y_{t+h}})'_0\in \T_{\gamma_t}\Y$ have a limit $y'^+_t$ in $\T_{\gamma_t}\Y$ as $h\downarrow 0$.
\end{proposition}
For us such concept will be useful in particular in connection with the well known first-order variation of the squared distance:
\begin{proposition}\label{cor:derd2}
Let $\Y$ be a  $\Cat\kappa$-space, $(y_t)$ an absolutely continuous curve and $z\in\Y$. Then:
\[
\frac{\d}{\d t}\tfrac12\sfd_\Y^2(y_t,z)=-\la y_t'^+,(\G_{y_t}^z)'_0\ra_{\gamma_t}\qquad a.e.\ t.
\]
\end{proposition}
To prove the above proposition, see e.g. \ \cite[Propositions 2.17 and 2.20]{DMGSP18}, one needs to introduce the notion of angle between geodesics and study its monotonicity properties, its behaviour along absolutely continuous curve and finally its connection with the inner product we introduced. Nevertheless, even if we omit the proof, in the sequel we shall use the following fact (see \cite[Lemma 2.19]{DMGSP18}): let $\Y$ be $\Cat\kappa$, $(y_t)$ be an absolutely continuous curve and $z \in \Y$. Then, for the time $t$ s.t. $\msp$ exists and it is positive, we have
\begin{equation}
\begin{split}
-\la \tfrac1h(\G_{y_t}^{y_{t+h}})'_0,(\G_{y_t}^z)'_0\ra_{y_t} &\le -\sfd_\Y(y_t,z)\frac{\sfd_\Y(y_{t+h},y_t)}h\cos(\angle_{y_t}^\kappa(y_{t+h},z)),\quad\forall h>0\text{ s.t. }y_{t+h}\in B_{\sfr_{y_t}}(y_t)\\
\liminf_{h\downarrow 0}-\cos(\angle_{y_t}^\kappa(y_{t+h},z))&=\liminf_{h\downarrow 0}\frac{\sfd_\Y(y_{t+h},z)-\sfd_\Y(y_t,z)}{h\msp},
\end{split} 
\label{eq:d2variation}
\end{equation}
where $\angle_{y_t}^\kappa(y_{t+h},z)$ is  the angle at $\bar{y}$ in $\M_k$ of the comparison triangle $\triangle^\kappa(\bar{y},\bar{y}_h,\bar{z})$. The first of these is an obvious consequence of the definition of $\angle_{y}^\kappa(z_1,z_2)$ together with the fact that $\kappa\mapsto \angle_{y}^\kappa(z_1,z_2)$, and thus $\kappa\mapsto -\cos(\angle_{y}^\kappa(z_1,z_2))$, is increasing, while the second one follows from the Taylor expansion of $\cos(\angle_{y}^\kappa(z_1,z_2))$ for $\sfd_\Y(y,z_1)$ small (notice that the explicit formula for $\cos(\angle_{y}^\kappa(z_1,z_2))$ in terms of $\sfd_\Y(y,z_1),\sfd_\Y(y,z_2),\sfd_\Y(z_1,z_2)$ can be obtained by the cosine rule).

\subsection{Weak convergence}
In this section, we recall the concept of weak convergence in a \Cat 0-space, highlighting the similarities with weak convergence on a Hilbert setting. 

Still, it is important to underline that although a well-behaved notion of `weakly converging sequence' exists, in \cite{Bac18} it is stressed that the existence of a well-behaved weak topology inducing such convergence is an open challenge. For the goal of this manuscript, here we just recall an operative  definition of weak convergence and its properties.
\bigskip

Let us first clarify the notion of \emph{semiconvexity} on a geodesic metric space.
\begin{definition}[Semiconvex function]
Let $\Y$ be a geodesic space and $\E \colon \Y \rightarrow \R \cup \{+\infty\} $. We say that $\E$ is \emph{$\lambda$-convex},   $\lambda \in \R$, if  for any geodesic $\gamma$ it holds
$$ \E(\gamma_t) \le (1-t)\E(\gamma_0) +t\E(\gamma_1) -\frac{\lambda}{2}t(1-t)\sfd_{\Y}^2(\gamma_0,\gamma_1)\qquad\forall t\in[0,1].$$
If $\lambda=0$, then we simply speak of convex functions. We shall denote by $D(\E)\subset\Y$ the set of $y$'s such that $ \E(y)<\infty$.
\end{definition}

Notice that, if $\Y$ is \Cat0 and $\E\colon\Y\to\R^+\cup\{+\infty\}$ is 2-convex and lower semicontinuous, then it admits a unique minimizer. To see this, we argue as for Proposition \ref{prop:proj} and prove that any minimizing sequence $(y_n)\subset\Y$ is Cauchy: let $I:=\inf \E\geq 0$, $y_{n,m}$ the midpoint of $y_n,y_m$ and notice that 
\[
I\leq \E(y_{n,m})\leq \frac12\big(\E(y_n)+\E(y_m)\big)-\frac14 \sfd^2_\Y(y_n,y_m)\qquad\forall n,m\in\N,
\]
so that rearranging and passing to the limit we get 
\[
\frac14\lims_{n,m\to\infty}\sfd^2_\Y(y_n,y_m)\leq \lims_{n,m\to\infty}\frac12\big(\E(y_n)+\E(y_m)\big)-I=0,
\]
giving the claim. The first example of $2$-convex functional we have encountered is the squared distance from a point in a \Cat0-space, as inequality \eqref{eq:cat0def} suggests. Hence, for $(y_n)\subset\Y$ be a bounded sequence, we can consider the mapping 
\[
\Y \ni y \mapsto \omega(y ; (y_n)) := \limsup_{n} \sfd_{\Y}^2(y,y_n).
\] 
and notice that, as a limsup of a sequence of $2$-convex and locally equiLipschitz functions, it is still $2$-convex and locally Lipschitz. By the above remark, it has a unique minimizer.
\begin{definition}[Asymptotic center and weak convergence]\label{def:weakconv}
Let $\Y$ be \Cat0-space and $(y_n)$ be a bounded sequence. We call the minimizer of $\omega(\cdot,(y_n))$ the \emph{asymptotic center} of $(y_n)$.

We say that a sequence $(y_n) \subset \Y$ \emph{weakly converges} to $y$, and write  $y_n \weakto y$, if  $y$ is the asymptotic center of every subsequence $(y_{n_k})$ of $(y_n)$.
\end{definition}
In analogy with the Hilbert setting, we shall sometimes say that $(y_n)$ converges strongly to $y$ if $\sfd_\Y(y_n,y)\to 0$. The main properties of weak convergence are collected in the following statement:
\begin{proposition}\label{prop:weakprop}
Let $\Y$ be a \Cat0-space. Then, the following holds:
\begin{itemize}
\item[i)] If $(y_n)$ converges to $y$ strongly, then it converges weakly.
\item[ii)] $y_n\to y$ if and only if $y_n\weakto  y$ and for some $z\in\Y$ we have $\sfd_\Y(y_n,z)\to \sfd_\Y(y,z)$.
\item[iii)] Any bounded sequence admits a weakly converging subsequence.
\item[iv)] If $C\subset \Y$ is convex and closed, then it is  sequentially weakly closed.
\item[v)] If $\E\colon\Y\to\R\cup\{+\infty\}$ is a convex and lower semicontinuous function, then it is sequentially weakly lower semicontinuous.
\end{itemize}
Moreover, at the tangent cone $ \T_y\Y$ at $y\in\Y$ (which is also a \Cat0-space by Theorem \ref{thm:tancat}) we also have
\begin{itemize}
\item[vi)] Let $(v_n),(w_n)\subset \T_y\Y$ be such that $v_n\to v$ and $w_n\weakto w$ for some $v,w\in \T_y\Y$. Then $\lims_{n\to\infty}\la v_n,w_n\ra \leq\la v,w\ra$.
\end{itemize}
\end{proposition}
\begin{proof} $(i)$ is obvious, as a strong limit is  trivially the asymptotic center of the full sequence. For $(ii),(iii),(iv)$ see  {\cite[Proposition 3.1.6]{Bac14}}, {\cite[Proposition 3.1.2]{Bac14}} and {\cite[Proposition 3.2.1]{Bac14}} respectively. $(v)$ follows trivially from $(iv)$ by considering the strongly closed and convex sublevels of $\E$. Finally, for $(vi)$ we let $C:=\sup_n|w_n|_y<\infty$ and notice that for every $\eps>0$ it holds
\[
\begin{split}
2\eps\la v_n,w_n\ra_y=\la v_n,2\eps w_n \ra_y&\leq |v_n|_y^2+|2\eps w|^2_y-\sfd_y^2(v,2\eps w_n)+\big(\sfd_y^2(v,2\eps w_n)-\sfd_y^2(v_n,2\eps w_n)\big)+4\eps^2C^2\\
&\leq |v_n|_y^2+|2\eps w|^2_y-\sfd_y^2(v,2\eps w_n)+4\eps C\sfd_y(v,v_n)(|v|_y+|v_n|_y)+4\eps^2C^2
\end{split}
\]
and that $\eps w_n\weakto \eps w$ (by \eqref{eq:norm}). Sending $n\to \infty$ and using the sequential weak lower semicontinuity of $\sfd_y^2(v,\cdot)$ (consequence of $(v)$) we obtain that
\[
2\eps\lims_{n\to\infty}\la v_n,w_n\ra_y\leq |v|_y^2+|2\eps w|^2_y-\sfd_y^2(v,2\eps w)+4\eps^2C^2=2\eps \la v,w\ra_y+4\eps^2C^2
\]
and the claim follows dividing by $\eps>0$ and letting $\eps\downarrow0$.
\end{proof}

\subsection{Geometric tangent bundle} \label{Sect2.geo}

In this section we briefly recall some concepts from \cite{DMGSP18} about the construction of the Geometric Tangent Bundle $\T_G\Y$ of a given \emph{separable} local $\Cat\kappa$-space $\Y$.  From now on, $\cB(\Y)$ is the Borel $\sigma$-algebra on $\Y$. As a set, the space $\T_G\Y$ is defined as
\[
\T_G\Y:=\big\{(y,v)\colon y\in\Y,\ v\in\T_y\Y\big\}.
\]
Such set is equipped with a $\sigma$-algebra $\cB(\T_G\Y)$, called Borel $\sigma$-algebra (with a slight abuse of terminology, because   there is no topology inducing it), defined as the smallest $\sigma$-algebra such that the following maps are measurable:
\begin{itemize}
\item[i)]  the canonical projection $\pi_\Y\colon\T_G\Y\to \Y$
\item[ii)] the maps $\pi_{\Y}^{-1}(B_{r_{\bar{y}}}\bar{y})\ni(y,v)\mapsto\la v,(\G^z_y)'_0\ra_y\in\R$ for every $\bar{y}\in \Y, z \in B_{r_{\bar{y}}}(\bar{y})$. 
\end{itemize}
It turns out that $\cB(\T_G\Y)$ is countably generated and that, rather than asking $(ii)$ for every $z\in\Y$, one can require it only for a dense set of points (notice that in the axiomatization chosen in \cite{DMGSP18} one speaks about the differential of the distance function rather than of scalar product with vectors of the form $(\G^z_y)'_0$, but the two approaches are actually trivially equivalent thanks to the explicit expression of the differential of the distance in terms of such scalar product which is hidden in Proposition \ref{cor:derd2}). We also recall that
\begin{equation}
\label{eq:normbor}
\text{the map $\T_G\Y\ni (y,v)\mapsto |v|_y\in\R$ is Borel.}
\end{equation}

A \emph{section} of $\T_G\Y$ is a map ${\sf s}\colon \Y\to\T_G\Y$ such that ${\sf s}_y\in \T_y\Y$ for every $\Y$. A section is said Borel if it is measurable w.r.t.\ $\cB(\Y)$ and $\cB(\T_G\Y)$. Among the various sections, \emph{simple} ones play a special role, similar to the one played by finite-ranged functions in the theory of Bochner integration: ${\sf s}$ is a simple section provided there are $(y_n)\subset\Y$, $(\alpha_n)\subset\R^+$ and $(E_n)$ Borel partition of $\Y$ such that $y_n \in B_{\sfr_y}(y)$ for every $y \in E_n$ and ${\sf s}\restr{E_n}=\alpha_n(\G_\cdot^{y_n})'_0$. If this is the case we write ${\sf s}=\sum_n\nchi_{E_n}\alpha_n(\G_\cdot^{y_n})'_0$, although the `sum' here is purely formal. The following basic result - obtained in \cite{DMGSP18} - will be  useful, we report the proof  for completeness:
\begin{proposition}\label{prop:sb}
Let $\Y$ be  separable and locally $\Cat\kappa$. Then, simple sections of $\T_G\Y$ as defined above are Borel.
\end{proposition}
\begin{proof}
It is sufficient to prove that for any given $\bar{y}\in\Y, z \in B_{r_{\bar{y}}}(\bar{y})$ and $\alpha\in\R^+$ the assignment $B_{r_{\bar{y}}}(\bar{y})\ni y\mapsto {\sf s}_y:= \alpha(\G^z_y)'_0$ is Borel and to this aim, by the very definition of $\cB(\T_G\Y)$, it is sufficient to check that $\pi_\Y\circ {\sf s}\colon\Y\to \Y$ is Borel - which it is, being this map the identity on $\Y$ - and, for any $w\in B_{r_{\bar{y}}}(\bar{y})$, the map $B_{r_{\bar{y}}}(\bar{y})\ni y\mapsto\la {\sf s}_y,(\G^w_y)'_0\ra_y$ is Borel.  Thus fix $w$ and notice that thanks to \eqref{eq:normbor} and to the definition of scalar product on $\T_y\Y$ to conclude it is sufficient to check that $y\mapsto\sfd_y({\sf s}_y,(\G^w_y)'_0)$ is Borel. We have
\[
\sfd_y({\sf s}_y,(\G^w_y)'_0)=\sfd_y(\alpha(\G^z_y)'_0,(\G^w_y)'_0)=\lim_{t\downarrow0} \frac{\sfd_\Y\big((\G^z_y)_{\alpha t},(\G^w_y)_t\big)}t.
\]
From the continuous dependence of geodesics on their endpoints we deduce that $y\mapsto \sfd_\Y\big((\G^z_y)_{\alpha t},(\G^w_y)_t\big)$ is a continuous function for every $t\in (0,1\wedge\alpha^{-1})$. The conclusion then follows from the fact that a pointwise limit of continuous functions is Borel.
\end{proof}
It has been proved in \cite{DMGSP18} that simple sections are dense among Borel ones (see also Lemma \ref{le:denssimp} below in the case $\X=\Y$ and $u={\rm Identity}$). Moreover, the operations on a single tangent space $\T_y\Y$ induce in a natural way operations on the space of Borel sections of $\T_G\Y$: these are Borel regular, as recalled in the next statement (see  \cite[Proposition 3.6]{DMGSP18}  for the proof).
\begin{proposition}\label{prop:bormap}
Let $\Y$ be separable and locally $\Cat\kappa$, ${\sf s},{\sf t}$ Borel sections of $\T_G\Y$ and $f\colon\Y\to\R^+$ Borel. Then, the maps from $\Y$ to $\R$ sending $y$ to $|{\sf s}_y|_y,\sfd_y({\sf s}_y,{\sf t}_y),\la {\sf s}_y,{\sf t}_y\ra_y$ are Borel and the sections $y\mapsto f(y) {\sf s}_y,{\sf s}_y\oplus {\sf t}_y$ are Borel as well.
\end{proposition}
%

\section{Gradient flows on {\sf CAT}$(\kappa)$-spaces}\label{Sect3}

\subsection{Metric approach}
We recall here the basic definitions and properties of gradient flows on locally  $\Cat\kappa$-spaces. We begin with the definition of (descending) slope $|\partial^-\E|$ of the functional $\E$: for $y\in D(\E)$ we put
\begin{equation}
\label{eq:defsl}
 \vert \partial^-\E\vert(y) := \limsup_{z \rightarrow y} \frac{(\E(y)-\E(z))^+}{\sfd_{\Y}(y,z)},
\end{equation}
and we denote the points where the slope is finite by $D(|\partial^-\E|)\subset D(\E)$. It is easy to prove that for $\lambda$-convex functionals, the slope admits the following `global' formulation (see \cite[Theorem 2.4.9]{AmbrosioGigliSavare08} for the proof):
\begin{lemma}
Let $\Y$ be a geodesic space and $\E \colon \Y \rightarrow \R \cup \{+\infty\}$ be $\lambda$-convex, $\lambda\in\R$, and lower semicontinuous. Then, for every $y \in D(\E)$,
$$ \vert \partial^- \E \vert(y) = \sup _{z \neq y} \left(\frac{\E(y)-\E(z)}{\sfd_{\Y}(y,z)} + \frac{\lambda}{2}\sfd_{\Y}(y,z)\right)^+.  $$
Moreover, $y \mapsto |\partial^-\E| (y)$ is a lower semicontinuous function. \label{lem:slopelsc}
\end{lemma}
We now come to various equivalent definitions of gradient flows on locally  \Cat{$\kappa$}-spaces. The equivalence between the first two notions below is due to the convexity assumption, while the equivalence of these with the EVI is due to  the geometric properties of $\Cat\kappa$-spaces, and in particular their Hilbert-like structure at small scales.
\begin{theorem}[Gradient flows on locally $\Cat\kappa$-spaces: equivalent definitions]\label{thm:GFdef}
Let $\Y$ be a locally $\Cat\kappa$-space, $\E\colon\Y\to\R\cup\{+\infty\}$ a $\lambda$-convex and lower semicontinuous functional, $\lambda \in \R$, $y\in\Y$ and $(0,\infty)\ni t\mapsto y_t\in\Y$ a locally absolutely continuous curve such that $y_t\to y$ as $t\downarrow0$. Then, the following are equivalent:
\begin{itemize}
\item[$(i)$] \textsc{Energy Dissipation Inequality} We have
\[
-\partial_t\E(y_t)\geq  \frac12|\dot y_t|^2+\frac12|\partial^-\E|^2(y_t)
\]
where the derivative in the left hand side is intended in the sense of distributions.
\item[$(ii)$] \textsc{Sharp dissipation rate} $t\mapsto \E(y_t)$ is locally absolutely continuous and 
\begin{equation}
\lim_{h \downarrow 0}\frac{\E(y_t)-\E(y_{t+h})}{h} =\vert\dot{y}^+_t\vert^2= \vert \partial^-\E\vert^2(y_t) \qquad \text{for every }  t>0, \label{eq:metconv}
\end{equation}
where $\vert\dot{y}^+_t\vert:=\lim_{h\downarrow0} \frac{\sfd_\Y(y_{t+h},y_t)}{h}$ is the right metric speed, which in this case exists for every $t>0$.
\item[$(iii)$] \textsc{Evolution Variational Inequality} For every $z\in\Y$ we have
\begin{equation}
\label{eq:evi}
\frac{\d}{\d t}\frac{\sfd^2_\Y(y_t,z)}2+\E(y_t)+\frac\lambda2\sfd_\Y^2(y_t,z)\leq \E(z)\qquad a.e.\ t>0.
\end{equation}
\end{itemize}
\end{theorem}
\begin{proof}
The fact that $(ii)$ implies $(i)$ is obvious. The converse implication has been proved in \cite{AmbrosioGigliSavare08} as a consequence of the so called \emph{strong upper gradient property} of the slope. The implication $(iii)\rightarrow(ii)$ is proved in \cite{MS20} (the argument in \cite{MS20} has been also reported in \cite{G11}). The fact that on  locally $\Cat\kappa$-spaces $(ii)$ implies $(iii)$ has also been proved in \cite{MS20} (see in particular Theorems 4.2 and 3.14 there). More precisely, in \cite{MS20} only the `global' case of $\Cat\kappa$-spaces has been considered, but the arguments there can be quickly adapted to cover our case by noticing that:
\begin{itemize}
\item[-] arguing as for the proof of \eqref{eq:subder} below, we see that  \eqref{eq:evi} holds at some $t$ if and only if it holds at $t$ for  $z$ varying only in a neighbourhood of $y_t$,
\item[-] property \eqref{eq:metconv} is local by nature,
\item[-] if $B\subset\Y$ is closed, convex and $\Cat\kappa$, then a curve $I\ni t\mapsto y_t\in B$ satisfies $(ii)$ (resp.\ $(iii)$) in $B$ if and only if it satisfies $(ii)$ (resp.\ $(iii)$) in $\Y$.
\end{itemize}
\end{proof}
A curve satisfying any of the equivalent conditions in this last theorem will be called \emph{gradient flow trajectory}. Moreover, we define the \emph{gradient flow map} ${\sf GF}^\E \colon (0,\infty) \times \Y \rightarrow \Y$ via ${\sf GF}_t^\E(y):= y_t$ for every $t \in (0,\infty), y \in \Y$, where, evidently, $y_t$ is the gradient flow trajectory starting at $y$ and associated to the functional $\E$ evaluated at time $t$. Some of their main properties are collected in the following statement:
\begin{theorem}[Gradient flows on locally $\Cat\kappa$-spaces: some basic properties]\label{thm:GF}
Let $\Y$ be a locally $\Cat\kappa$-space, $\E\colon\Y\to\R\cup\{+\infty\}$ a $\lambda$-convex and lower semicontinuous functional. Then, the following holds:
\begin{itemize}
\item[\textopenbullet]  \textsc{Existence} \begin{center}For every $y\in \overline{D(\E)}$ there exists a gradient flow trajectory for $\E$ starting from $y$.\end{center}
\item[\textopenbullet]  \textsc{Uniqueness and $\lambda$-contraction} \begin{center}For any two gradient flow trajectories  $(y_t),(z_t)$ starting from $y,z$ respectively we have
\begin{equation}
\label{eq:contr}
\sfd_\Y(y_t,z_t)\leq e^{-\lambda (t-s)}\sfd_\Y(y_s,z_s)\qquad\forall t\ge s >0
\end{equation}
\end{center}
\item[\textopenbullet]  \textsc{Monotonicity properties}\ For $(y_t)$ gradient flow trajectory for $\E$ starting from $y$       we have that 
         \begin{center}\label{eq:s5}
        $t\mapsto y_t$ is locally Lipschitz in $(0,+\infty)$ with values in $D(|\partial^-\E|)\subset D(\E)$,
        \begin{align}
            &t \mapsto \E(y_t) \text{ is nonincreasing in $ [0,+\infty)$}, \nonumber\\
            &t \mapsto e^{\lambda t}|\partial^-\E|({y_t}) \text{ is nonincreasing in $ [0,+\infty)$}. \label{eq:regul}
        \end{align}
         \end{center}
\end{itemize}
\end{theorem}
\begin{proof} In the \Cat0 case, existence of a limit of the so-called minimizing movements scheme in this setting has been proved in \cite{Mayer98} and \cite{Jost98}. The fact that the limit curve obtained in this way satisfies the EVI condition has been proved in  \cite{AmbrosioGigliSavare08}. The contractivity property, also at the level of the discrete scheme, has been proved in  \cite{Mayer98} and \cite{Jost98} (at least in the case $\lambda=0$, the general case can be found e.g.\ in  \cite{AmbrosioGigliSavare08} as a consequence of the EVI condition). Then, uniqueness is directly implied by \eqref{eq:contr} and the last claims are a consequence of \eqref{eq:metconv} and the contraction property.

The  $\Cat\kappa$ case has been treated in \cite{OP17}, at least under some compactness assumptions on the sublevels of the functional. Such compactness assumption has been removed in \cite{MS20}. Finally, the case of locally $\Cat\kappa$ spaces can be dealt with as in the proof of Theorem \ref{thm:GFdef} above.
\end{proof}
Finally, we conclude the section with an \emph{a priori} estimate, a variant of the ones investigated in \cite{MS20}, concerning contraction properties along the gradient flow trajectories at different times. The proof is inspired by the one of \cite[Lemma 2.1.4]{PE13} in the context of CBB-spaces.
\begin{lemma}[A priori estimates] \label{lem:apriori}
Let $\Y$ be locally $\Cat\kappa$ and $\E \colon \Y\rightarrow [0,\infty]$ be a $\lambda$-convex and lower semicontinuous functional, $\lambda \in \R$. Let $y,z \in \Y$ and consider the gradient flow trajectories $(y_t),(z_t)$ associated with $\E$.

Then, for any $t \ge s > 0$, it holds
\begin{equation}
\label{eq:apriorigf}
 \begin{split}
\sfd^2_{\Y}(y_{t},z_{s}) \le  e^{-2\lambda s}\Big(  \sfd^2_\Y(y,z) +& 2(t-s)(\E(z)-\E(y)) \\
& + 2|\partial^- \E|^2(y)\int_0^{t-s}\theta_\lambda(r)\,\d r -\lambda\int_0^{t-s} \sfd_\Y^2(y_r,z)\, \d r\Big),
\end{split}
\end{equation}
where $\theta_\lambda(t):= \int_0^t e^{-2\lambda r}\, \d r$.
\end{lemma}
\begin{proof}
We start fixing $t>0$. First, we notice that, in light of $(ii)$ of Theorem \ref{thm:GFdef} and the basic properties in Theorem \ref{thm:GF}, we have for any $r >0$ (and not a.e. $r$),
\[ 
-\E(y_r) + \E(y) = \int_0^r  e^{-2\lambda q}e^{+2\lambda q}|\partial^-\E|^2(y_q)\, \d q \le|\partial^- \E|^2(y)\theta_\lambda(r). 
\]
Thus, we can integrate from $0$ to $t$ the EVI condition \eqref{eq:evi} to get
\[
\begin{split}
\frac12(\sfd_\Y^2(y_t,z) - \sfd_\Y^2(y,z) ) &\le \int_0^t \E(z)-\E(y_r) -\frac{\lambda}{2}\sfd_\Y^2(y_r,z)\, \d r  \\
&\le  t(\E(z)-\E(y))+ |\partial^- \E|^2(y)\int_0^t \theta_\lambda(r)\, \d r -\frac{\lambda}{2}\int_0^t \sfd_\Y^2(y_r,z)\, \d r.
\end{split}
\]
Finally, for general $t \ge s >0 $, we can reduce to above case by appealing to property \eqref{eq:contr}. 
\end{proof}

\subsection{The object $-\partial^-\E(y)$}

In this section we introduce the key object  $-\partial^-\E(y)$ of this manuscript associated to a semiconvex and lower semicontinuous functional $\E$ over a local $\Cat\kappa$ space. As the notation suggests, and as will be clear from Definition \ref{def:md}, for functionals on Hilbert spaces this corresponds to $\{-v:v\in\partial^-\E(y)\}$.

We start recalling the following well known fact:
\begin{proposition}[Metric projection]\label{prop:proj}
Let $\Y$ be a $\Cat0$-space and  $C \subset \Y$ be a closed convex subset. Then, for every $y\in\Y$, there is a unique $ \Pr_{C}(y) \in C$, called \emph{metric projection} of $y$ onto $C$, such that $\sfd_{\Y}(y,\Pr_{C}(y)) = \inf_{C}\sfd_{\Y}(y,\cdot)$.

\end{proposition}
\begin{proof} Since the function to be minimized is continuous and $C$ closed, it is sufficient to prove that any minimizing sequence $(c_n)$ for  $I:= \inf_{c \in C} \sfd_{\Y}^2(c,y)$ (which is equivalent to be minimizing for $\inf_C\sfd_\Y(y,\cdot)$) is Cauchy. Fix such sequence  and, for every $n,m \in \N$, let  $c_{m,n}$ be the mid-point between $c_n$ and $c_m$. Observe that since $C$ is convex, $c_{n,m}$ belongs to $C$ and thus  is a competitor for the minimization problem. Condition (\ref{eq:cat0def}) therefore implies
$$ I \le \sfd_{\Y}^2(c_{n,m},y) \le \frac{1}{2}\sfd_{\Y}^2(c_n,y) + \frac{1}{2}\sfd_{\Y}^2(c_m,y) - \frac{1}{4}\sfd_{\Y}^2(c_n,c_m),$$
for every $n,m \in \N$.
Rearranging terms, and taking the limsup as $n,m$ go to infinity we observe
$$\limsup_{n,m \rightarrow +\infty} \frac{1}{4}\sfd_{\Y}^2(c_n,c_m) \le \limsup_{n,m \rightarrow +\infty} \frac{1}{2}\sfd_{\Y}^2(c_n,y) + \frac{1}{2}\sfd_{\Y}^2(c_m,y) -I =0, $$
i.e.\ $(c_n)$ is Cauchy, as desired.
\end{proof}
We remark that the metric projection can be also shown to be $1$-Lipschitz and to satisfy a `Pythagoras' inequality' (see \cite[Theorem 2.1.12]{Bac14}), but we will not make use of this fact.bFinally, we are ready to give an effective definition of (opposite of the) subdifferential of $\E$ as a subset of the tangent cone.
\begin{definition}[Minus-subdifferential]\label{def:md}
Let $\Y$ be locally $\Cat\kappa$, $\E \colon \Y \rightarrow \R \cup \{+\infty\}$ be a $\lambda$-convex and lower semicontinuous functional, $\lambda \in \R$, and $y \in D(\E)$. We define the \emph{minus-subdifferential} of $\E$ at $y$, denoted by $-\partial^-\E(y)$, as the collection of $ v \in \T_y\Y$ satisfying the subdifferential inequality 
\[
\E(y) - \la v,\gamma'_0\ra_y+\frac{\lambda}{2}\sfd_\Y^2(y,z) \le \E(z),
\]
for every $z \in \Y,$ and some geodesic $\gamma$ from $y$ to $z$. Moreover, by $D(-\partial^-\E)$, we denote the collection of $y \in \Y$ for which $-\partial^-\E(y)\neq \emptyset$.
\label{def:subdiff}
\end{definition}
Notice that $v\in-\partial^-\E(y)$ if and only if
\begin{equation}
\label{eq:subder}
- \la v,\gamma'_0\ra_y\leq \lim_{t\downarrow0}\frac{\E(\gamma_t)-\E(y)}t\qquad\forall z\in\Y,\ \text{for some geodesic } \gamma \text{ from }y \text{ to }z.
\end{equation}
so that in particular the definition of $-\partial^-\E(y)$ does not depend on $\lambda$. Indeed the `if' is obvious by $\lambda$-convexity while for the `only if' we apply the defining inequality with $z_t:=\gamma_t$ in place of $z$ and, for $t$ small enough, rearrange to get
\[
-\la v,(\G_y^{z_t})'_0\ra_y+\frac{\lambda}{2}\sfd_\Y^2(y,z_t)\leq \E(z_t)-\E(y)
\]
so that the conclusion follows noticing that $\sfd_\Y^2(y,z_t)=t^2\sfd_\Y^2(y,z)$, $(\G_y^{z_t})'_0=t\gamma'_0$ (because for $t\ll1$ the geodesic from $y$ to $z_t$ in unique), then dividing by $t$ and letting $t\downarrow0$. The same arguments also show that both in Definition \ref{def:md} and in \eqref{eq:subder} we can take $\gamma$ to be \emph{any} geodesic from $y$ to $z$.

It is also worth to point out that
\begin{equation}
\label{eq:sub0}
\begin{split}
\text{For $\E$ convex and lower semicontinuous we have that:}\\
\text{$x$ is a minimum point for $\E$ if and only if $0\in-\partial^-\E(x)$.}
\end{split}
\end{equation}
The proof of this fact being obvious.

\begin{remark}\label{re:mm}{\rm It would certainly be possible to define the analogous notion of subdifferential $\partial^-\E$ by replacing $ - \la v,\gamma'_0\ra_y$ with $\la v,\gamma'_0\ra_y$ in the defining formula, however, since the tangent cone is only a cone and not a space, there is no obvious relation between the two definitions.

For our purposes, $-\partial^-\E$ is the correct object to work with because, as discussed in the introduction, we aim at showing the existence of the Laplacian of a \Cat0-valued Sobolev map by looking at the gradient flow of the Korevaar-Schoen energy $\E^\sfKS$, thus we notice on one hand that, by definition and imitating what happens in the smooth category, the Laplacian of $u$ has to be introduced as (the element of minimal norm in) $-\partial^-\E^\sfKS(u)$, and on the other one that in the gradient flow equation \eqref{eq:gfi} it is $-\partial^-\E$ who appears.

In this direction, it is interesting to point out that the classical  procedure of minimizing 
\[
y\quad\mapsto\quad\E(y)+\frac{\sfd^2_\Y(y,\bar y)}{2\tau},
\]
which is the cornerstone of most existence results about gradient flows in the metric setting (see e.g.\ \cite{AmbrosioGigliSavare08}), produces a (unique, if $\tau>0$ is small enough) point $y_\tau$ for which we have $\frac1\tau(\G_{y_\tau}^{\bar y})'_0\in\partial^-\E(y_\tau)$. In particular it gives no informations  about whether $-\partial^-\E(y_\tau)$ is not empty. In our approach this latter fact, and the related one   that the slope at $y$ coincides with the norm of the element of least norm in $-\partial^-\E(y)$, will be a consequence of the fact that gradient flow trajectories satisfy an analogue of \eqref{eq:gfi}, see Theorem \ref{thm:rightD}.
\fr }
\end{remark}
It will be important to know that in $-\partial^-\E(y)$ there is always an element of minimal norm:
\begin{proposition}
Let $\Y$ be a locally $\Cat \kappa$-space, $\E\colon \Y \rightarrow \R \cup \{+\infty\}$ be a $\lambda$-convex and lower semicontinuous functional, $\lambda \in \R$, and $y \in \Y$. Then, $- \partial^-\E(y)$ is a closed and convex subset of $\T_y\Y$. In particular, if this set is not empty, the optimization problem $$\inf_{v \in -\partial^-\E(y)} \vert v \vert_y$$ admits a unique minimiser. 
\end{proposition}
\begin{proof} Recalling that $\T_y\Y$ is \Cat0, by Proposition  \ref{prop:proj} the existence of a unique minimizer in $-\partial^-\E(y)$ for the norm, i.e.\ of a unique metric projection of $0_y$ onto $-\partial^-\E(y)$, will follow once we show that $-\partial^-\E(y)$ is closed and convex. 

The fact that it is closed follows from the definition and the consideration already stated in Proposition \ref{prop:hilbertine} that the scalar product $\la\cdot,\cdot\ra_y$ is continuous on $\T_y\Y$. The convexity follows from the inequality
\[
-\la(\G_{v_1}^{v_2})_t,w\ra_y\leq -(1-t)\la v_1,w\ra_y-t\la v_2,w\ra_y\qquad\forall v_1,v_2,w\in\T_y\Y,\ t\in[0,1],
\]
which is a direct consequence of \eqref{eq:prhom} and \eqref{eq:concav}.
\end{proof}

\subsection{Subdifferential formulation}
Here we prove the main results of this note, namely Theorem \ref{thm:rightD} and Corollary \ref{cor:eqfor} below. We shall use the following preliminary result (notice that the fact that equality holds in \eqref{eq:sbn} will be obtained in \eqref{eq:s3}):
\begin{proposition}\label{prop:slopebnd}
Let $\Y$ be locally $\Cat\kappa$ and $\E \colon \Y \rightarrow \R \cup \{+\infty\}$ be a $\lambda$-convex and lower semicontinuous functional, $\lambda\in\R$. Then, for every $y\in D(-\partial^-\E)$, we have
\begin{equation}
\label{eq:sbn}
|\partial^-\E|(y)\leq\inf_{v\in -\subd y} |v|_y.
\end{equation}
In particular, $D(-\partial^-\E)\subset D(|\partial^-\E|) $.
\end{proposition}
\begin{proof} Let  $v \in -\partial^-\E (y)$ and notice that
$$ \E(y)-\E(z) +\frac{\lambda}{2}\sfd_{\Y}^2(y,z) \le \vert \la v,(\G_y^z)'_0\ra_y \vert \stackrel{\eqref{eq:CS}}\le \vert v\vert_y \sfd_{\Y}(y,z),\qquad\forall z\in\Y $$ 
which in turns implies
$$ \left(\frac{\E(y)-\E(z)}{\sfd_{\Y}(y,z)} +\frac{\lambda}{2}\sfd_{\Y}(y,z)\right)^+\le \vert v \vert_y \qquad \forall z \in \Y,\ z\neq y.$$
Taking the supremum over $z \neq y$ and recalling Lemma \ref{lem:slopelsc} we conclude.
\end{proof}

We now come to the main result of this manuscript, namely the existence of right incremental ratios of the flow for all time. 
\begin{theorem}[Right derivatives of the flow]\label{thm:rightD}
Let $\Y$ be locally $\Cat\kappa$ and $\E \colon \Y \rightarrow \R \cup \{+\infty\}$ be a $\lambda$-convex and lower semicontinuous functional, $\lambda\in\R$. Let $y \in \overline{D(\E)}$, and $(y_t)$ be the gradient flow trajectory starting from $y$ (recall Theorem \ref{thm:GF}). 

Then, for \emph{every} $t>0$, the right `difference quotient' $\frac{1}{h}(\G_{y_t}^{y_{t+h}})'_0$ strongly converges to the element of minimal norm in $-\partial^-\E({y_t})\subset \T_{y_t}\Y$ (i.e.\ to $\Pr_{-\partial^-\E({y_t})}(0_{y_t})$) as $h$ goes to $0^+$. The same holds for $t=0$ if (and only if) we have $y \in D(|\partial^-\E|)$. 

Moreover, $D(-\partial^-\E)=D(|\partial^-\E|)$ and the identity
\begin{equation}
\label{eq:s3}
|\partial^-\E|(y)=\min_{v\in -\partial^-\E(y)}|v|_y\qquad\forall y\in\Y
\end{equation}
holds, where as customary the minimum of the empty set is declared to be $+\infty$. In particular, $D(-\partial^-\E)$ is dense in $D(\E)$.
\end{theorem}
\begin{proof}
By the semigroup property ensured by the uniqueness of gradient flow trajectories and taking into account that $y_t\in D(|\partial^-E|)$ for every $t>0$ (recall \eqref{eq:s5}), it suffices to show the claim for $t=0$ under the assumption $y \in D(|\partial^-\E|)$. Suppose $y$ is not a minimum point for $\E$, otherwise there is nothing to prove. In particular, $(ii)$ of Theorem \ref{thm:GFdef} ensures that $|\dot{y}_0|$ exists and it is positive. Also, notice that  the continuity at time $t=0$ of the gradient flow trajectory ensures that for $\epsilon>0$ sufficiently small we have $y_h\in B_{\sfr_y}(y)$ for every $h \in (0,\epsilon)$. In particular for such $h$ the tangent vector $v_h:=\tfrac1h(\G_{y}^{y_{h}})'_0\in \T_y\Y$ is well defined and the statement makes sense. Fix such $\epsilon>0$.

\noindent\underline{\textsc{Step 1}} For every $h\in(0,\epsilon)$  we have
\begin{equation}
\label{eq:s2}
|v_h |_y=\frac{\sfd_{\Y}(y_h,y)}{h} \le  \fint_0^h \vert \dot{y}_t\vert \, \d t \stackrel{\eqref{eq:metconv}}=  \fint_0^h |\partial^-\E|(y_t) \, \d t \stackrel{\eqref{eq:regul}}\le |\partial^-\E| (y)\fint_0^h e^{-\lambda t} \, \d t.
\end{equation}
Hence  $\sup_{h\in(0,\epsilon)}|v_h|_y<\infty$, therefore point $(iii)$ of Proposition \ref{prop:weakprop} gives that for every sequence $h_n\downarrow0$ there is a subsequence, not relabelled, such that $v_{h_n} \weakto v$ for some  $v \in \T_y\Y$.

Fix such sequence and such weak limit $v$. To conclude it is sufficient to prove that the convergence is strong and that $v$ is the element of minimal norm in $-\partial^-\E(y)$, as this in particular grants that the limit is independent on the particular subsequence chosen.

\noindent\underline{\textsc{Step 2}} We claim that $v \in -\partial^-\E( {y})$. To see this, integrate \eqref{eq:evi} from $0$ to $h$ and divide by $h$ to obtain
\[
\frac{\sfd_{\Y}^2(y_h,z)-\sfd_{\Y}^2(y,z)}{2h}+ \fint_0^h \E(y_t)+\frac{\lambda}{2}\sfd_{\Y}^2(y_t,z) \, \d t \leq \E(z)\qquad\forall z\in\Y, \ h\in(0,\epsilon).
\]
Letting $h=h_n\downarrow0$ and recalling that $\E$ is lower semicontinuous we deduce that
\begin{equation}
\label{eq:s1}
\limi_{n\to\infty}\frac{\sfd_{\Y}^2(y_{h_n},z)-\sfd_{\Y}^2(y,z)}{2h_n}+\E(y)+\frac{\lambda}{2}\sfd_{\Y}^2(y,z)\leq \E(z)\qquad\forall z\in\Y.
\end{equation}
Next, fix $z \in \Y$, let $\gamma \in \Geo_y\Y$ with $\gamma_1=z$, denote $z_s:= \gamma_s$ and notice that, for $s$ sufficiently small, $z_s,y_{h_n} \in B_{\sfr_y}(y)$. Now \eqref{eq:d2variation} yields
\[
\begin{split}
\liminf_{n\rightarrow\infty }-\la v_{h_n} ,(\G_y^{z_s})'_0\ra_y &\le \sfd_\Y(y,z)\liminf_{n\rightarrow\infty}\frac{\sfd(y_{h_n},y)}{h_n}\frac{\sfd_\Y(y_{h_n},z_s)-\sfd_\Y(y,z_s)}{|\dot{y}_0|h_n}\\
&=\liminf_{n\rightarrow\infty}\frac{\sfd^2_\Y(y_{h_n},z_s)-\sfd^2_\Y(y,z_s)}{2h_n},
\end{split}
\]
having used the fact that $\liminf_na_nb_n = a\liminf_nb_n$ if $\lim_na_n=a>0$ and $(a_n),(b_n)\subset\R$ are bounded, and a chain rule argument in the last equality. Thus, recalling the weak upper semicontinuity of the scalar product proved in point $(vi)$ of Proposition \ref{prop:weakprop} we  get
\[
\limi_{n\to\infty}\frac{\sfd_{\Y}^2(y_{h_n},z_s)-\sfd_{\Y}^2(y,z_s)}{2h_n}\geq -\la v,(\G_y^{z_s})'_0\ra_{y}.
\]
Now, combine with \eqref{eq:s1} to get 
\[ \E(y) -\la v,(\G_y^{z_s})'_0\ra_{y} + \frac{\lambda}{2}\sfd^2_\Y(z_s,y) \le \E(z_s) \le (1-s)\E(y) + s\E(z)-\frac{\lambda}{2}s(1-s)\sfd^2_\Y(z,y). \]
Finally, using that $(\G_y^{z_s})'_0 = s\gamma'_0$, $ \sfd^2_\Y(z_s,y)=s^2\sfd_\Y^2(y,z)$ and \eqref{eq:prhom}, we can rearrange terms and take the limit as $s\downarrow 0$ to get 
\[ \E(y) -\la v,\gamma'_0\ra_{y} + \frac{\lambda}{2}\sfd^2_\Y(z,y) \le \E(z)\qquad \text{ for every } \gamma \text{ geodesic from $y$ to $z$}.\]
Given that $z$ was arbitrary, we conclude.

\noindent\underline{\sc Step 3} Since $|\cdot|_y^2:\T_y\Y\to\R$ is convex and continuous, by point $(v)$ of Proposition \ref{prop:weakprop} we  get
\[
\begin{split}
|v|^2_y\leq\limi_{n\to\infty}|v_{h_n}|^2_y\leq\lims_{n\to\infty}|v_{h_n} |^2_y\stackrel{\eqref{eq:s2}}\leq |\partial^-E|^2(y)\stackrel{\eqref{eq:sbn}}\leq\inf_{w\in-\partial^-\E(y)}|w|^2_y\leq |v|^2_y,
\end{split}
\]
and thus all the inequalities must be equalities. This proves at once the strong convergence of $(v_{h_n})$ to $v$ (by the convergence of norms and point $(ii)$ of Proposition \ref{prop:weakprop}) and that $v$ is the element of minimal norm in $-\partial^-\E(y)$.

The argument also proves that if $y\in D(|\partial^-\E|)$, then $y\in D(-\partial^-\E)$ and in this case the equality in \eqref{eq:s3} holds. Taking into account Proposition \ref{prop:slopebnd} we conclude that $D(|\partial^-\E|)=D(-\partial^-\E)$ that  \eqref{eq:s3} holds for every $y\in\Y$, as desired.

The last claim then follows from the existence of gradient flow trajectories starting from points in $D(\E)$ (Theorem \ref{thm:GF}) and \eqref{eq:metconv}.
\end{proof}
As a direct consequence of the above result, we see that we can characterize gradient flow trajectories by means of the classical differential inclusion $x_t'\in-\partial^-\E(x_t)$ which can be used to define such evolution in the Hilbert setting:
\begin{corollary}\label{cor:eqfor}
Let $\Y$ be locally $\Cat\kappa$ and $\E \colon \Y \rightarrow \R \cup \{+\infty\}$ be a $\lambda$-convex and lower semicontinuous functional, $\lambda\in\R$. Let $y \in \overline{D(\E)}$, and $(0,\infty)\ni t\mapsto y_t\in D(\E)$ be a locally absolutely continuous curve. Then, the following are equivalent:
\begin{itemize}
\item[(i)] $(y_t)$ is  a gradient flow trajectory for $\E$ starting from $y$, i.e.\ satisfies any of the three equivalent conditions in Theorem \ref{thm:GFdef}.
\item[(ii)] The right derivative $y'^+_t$ exists for every $t>0$ and
\[
\left\{\begin{array}{l}
y'^+_t\in-\partial^-\E(y_t)\quad\forall t>0\quad \text{ and is the element of minimal norm},\\
\displaystyle{\lim_{t\downarrow0}y_t}=y.
\end{array}
\right.
\]
If $y\in D(|\partial^-\E|)=D(-\partial^-\E)$ then the above holds also at $t=0$.
\item[(iii)] It holds
\[
\left\{\begin{array}{l}
y'^+_t\in-\partial^-\E(y_t)\quad a.e.\ t>0,\\
\displaystyle{\lim_{t\downarrow0}y_t}=y.
\end{array}
\right.
\]
\end{itemize}
\end{corollary}
\begin{proof}
The implication $(i)\Rightarrow(ii)$ is proved in Theorem \ref{thm:rightD} above and the one $(ii)\Rightarrow(iii)$ is obvious. The fact that $(iii)$ implies $(i)$ (in the form of the Evolution Variation Inequality) is a direct consequence of Proposition \ref{cor:derd2} (applied in a $\Cat\kappa$ neighbourhood of $y_t$ in combination with arguments  similar to those outlined in the proof of Theorem \ref{thm:GFdef}  to cover the case of a local $\Cat\kappa$ space) and the definition of $-\partial^-\E$.
\end{proof}
\begin{remark}{\rm
In the setting of Alexandrov geometry it is more customary to study the gradient flow of semi\emph{concave} functions ${\sf F}$, thus studying (a properly interpreted version of) $y_t'\in\partial^+{\sf F}$. Let $\E$ be semiconvex on a $\Cat\kappa$-space $\Y$ and put ${\sf F}:=-\E$. Then it is clear that the slope $|\partial^-\E|$ as we defined it coincides with the \emph{absolute gradient} $|\nabla{\sf F}|$ as defined in \cite[Definition 4.1]{Lyt05}, therefore, taking into account the characterization \eqref{eq:metconv}, we see that up to a different choice of parametrization, our notion of gradient flow trajectory coincides with the one of gradient-like curve studied in \cite[Definition 6.1]{Lyt05}.

The property $\frac{\d}{\d t^+}{\sf F}(y_t)=-\frac{\d}{\d t^+}\E(y_t)=|\partial^-\E|^2(y_t)=|v_t|^2_{y_t}$, where $v_t\in-\partial^-\E(y_t)$ is the element of minimal norm, together with the existence of the right derivative of $y_t$ and the characterization \eqref{eq:subder} show that the element of minimal norm in $-\partial^-\E(y)$ coincides with  $\nabla {\sf F}(y)$ as defined in \cite[Definition 11.4.1]{AKP19} on spaces with curvature bounded from \emph{below}.  This shows that our `differential' perspective on gradient flows is compatible with the one studied in \cite{AKP19} on CBB spaces.
}\fr\end{remark}

\section{Laplacian of \Cat0-valued maps}\label{Sect4}

\subsection{Pullback geometric tangent bundle} \label{Sec2.uTGY}

\subsubsection{The general non-separable case}

 For the purpose of this manuscript, a \emph{metric measure space} $(\X,\sfd,\mm)$ is always intended to be given by: a complete and separable metric space $(\X,\sfd)$ equipped with a non-negative and non-zero Borel measure giving finite mass to bounded sets. In some circumstances we shall add further assumptions on $\X$, typically in the form a $\RCD(K,N)$ condition.

Thus  let us fix a pointed \Cat0 space $(\Y,\sfd_\Y,\bar y)$, a metric measure space $(\X,\sfd,\mm)$  and an open subset  $\Omega\subset\X$. We recall that the space $L^0(\Omega,\Y)$ is the collection of all Borel maps $u\colon\Omega\to\Y$ which are essentially separably valued (i.e.\ for some separable subset $\tilde\Y\subset\Y$ we have $\mm(u^{-1}(\Y\setminus\tilde\Y))=0$), where two maps agreeing $\mm$-a.e.\ are identified. Then  $L^2(\Omega,\Y_{\bar y})\subset L^0(\Omega,\Y)$ is collection of those (equivalence classes of maps) $u$ such that $\int_\Omega\sfd^2_\Y(u(x),\bar y)\,\d\mm(x)<\infty$. The space $L^2(\Omega,\Y_{\bar y})$ comes naturally with the distance
\[
\sfd_{L^2}^2(u,v):=\int_\Omega\sfd^2_\Y(u(x),v(x))\,\d\mm(x)
\]
and by standard means one sees that with such distance the space is complete and that finite-ranged maps are dense.  Moreover, for $u,v\in L^2(\Omega,\Y_{\bar y})$ a direct computation shows that $t\mapsto (\G^v_u)_t\in L^2(\Omega,\Y_{\bar y})$, where $(\G^v_u)_t(x):=(\G^{v(x)}_{u(x)})_t$, is a geodesic from $u$ to $v$ (the fact that $(\G^v_u)_t\colon\Omega\to\Y$ is Borel follows from the continuous dependence of the  - unique - geodesics on $\Y$ w.r.t.\ the endpoints). Also, by appealing to the equivalent characterization \eqref{eq:cat0def} of \Cat0-spaces, the computation
\[
\begin{split}
\sfd_{L^2}^2( (\G^v_u)_t,w)&=\int \sfd^2_\Y((\G^{v(x)}_{u(x)})_t,w(x))\,\d\mm(x)\\
&\stackrel{\eqref{eq:cat0def}}\leq \int (1-t)\sfd_\Y^2(u(x),w(x))+t\sfd_\Y^2(v(x),w(x))-t(1-t)\sfd_\Y^2(u(x),v(x))\,\d\mm(x)\\
&=(1-t)\sfd_{L^2}^2( u,w)+t\sfd_{L^2}^2( v,w)-t(1-t)\sfd^2_{L^2}(u,v)
\end{split}
\]
valid for any $w\in L^2(\Omega,\Y_{\bar y})$ and every $t\in[0,1]$, reveals that $L^2(\Omega,\Y_{\bar y})$ is a \Cat0-space as well and thus $\G^v_u$ is the only geodesic from $u$ to $v$.

In particular, given $u\in L^2(\Omega,\Y_{\bar y})$ we have a well defined tangent cone $\T_u L^2(\Omega,\Y_{\bar y})$ containing what we may think of as the set of `infinitesimal variations' of $u$. Intuitively, these variations should correspond to a collection, for $\mm$-a.e.\ $x\in\Omega$, of a variation of $u(x)\in\Y$, i.e.\ to a collection of elements of $\T_{u(x)}\Y$.

We now want to make this intuition rigorous and, due to the fact that \Cat0-spaces are typically studied in non separable environments, we first discuss this case, postponing to the next sections the separable case and its relations with the Borel structure on $\T_G\Y$ seen in  Section \ref{Sect2.geo}. Fix $u \in L^2(\Omega,\Y_{\bar{y}})$ and a Borel representative of it, which by abuse of notation we shall continue to denote by $u$. By  $u^*\T_G\Y$ we intend the set
\[
u^*\T_G\Y:=\big\{(x,y,v)\colon x\in\Omega,\ y=u(x),\ v\in \T_y\Y\big\}\subset \X\times\T_G\Y
\]
\begin{figure}[!h]
    \centering
    \includegraphics[scale=.7]{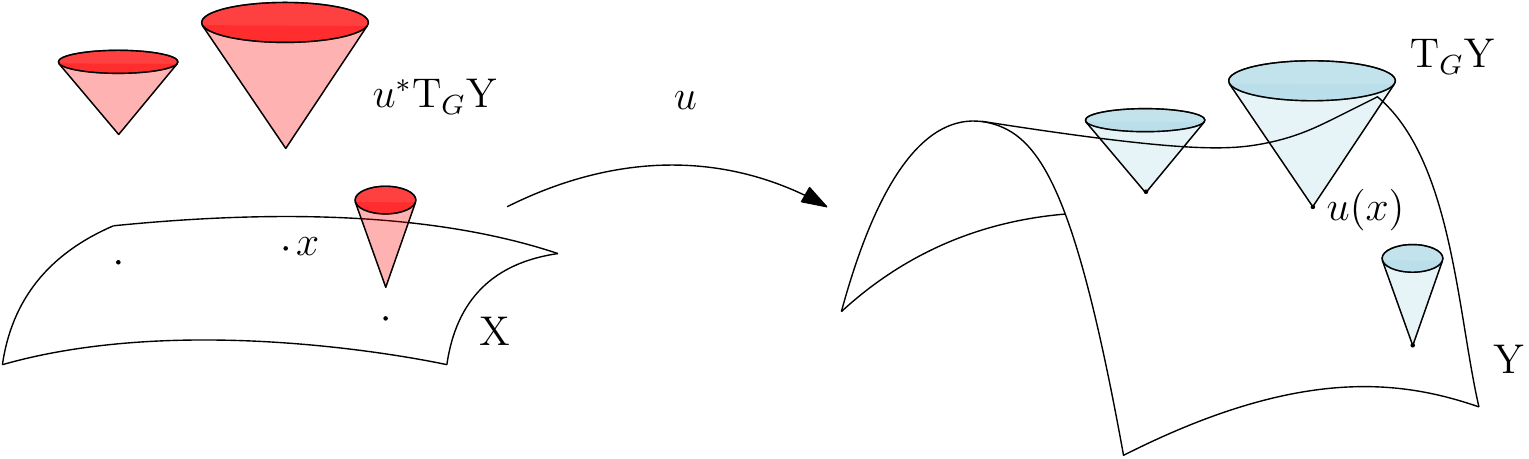}
    \caption{Pullback geometric tangent bundle $u^*\T_G\Y$ via $u \colon \X\to \Y$.}
    \label{fig:uTGY}
\end{figure}
A section  of $u^*\T_G\Y$  is a map ${\sf S}\colon\Omega\to u^*\T_G\Y$ such that $\pi_\X({\sf S}(x))=x$, where $\pi_\X\colon u^*\T_G\Y\to\X$ is the  canonical projection.   Given such a section ${\sf S}$ we write ${\sf S}(x)=(x,u(x),{\sf S}_x)$ for any $x\in\Omega$. We shall denote by ${\sf 0}$ the zero section defined by ${\sf 0}_x:=0_{u(x)}\in \T_{u(x)}\Y$.

Then given another $v\in L^2(\Omega,\Y_{\bar{y}})$ and a Borel representative of it, still denoted by $v$,  and $\alpha\geq 0$, we can consider the section ${\sf S}$ of $u^*\T_G\Y$  given by $x\mapsto (x,u(x),\alpha(\G_{u(x)}^{v(x)})'_0)$. We then have the following simple and useful lemma.
\begin{lemma}\label{le:l2s}
Let $(\Y,\sfd_\Y,\bar y)$ be a pointed   \Cat0-space, $(\X,\sfd,\mm)$ a metric measure space, $\Omega\subset\X$ an open subset, $u,v^1,v^2\colon \Omega\to\Y$ Borel representatives of maps in $L^2(\Omega,\Y_{\bar y})$. Also, let  $\alpha^1,\alpha^2\in\R^+$ and consider the sections ${\sf S}^i$ of $u^*\T_G\Y$ given by ${\sf S}^i_x:=\alpha_i(\G_{u(x)}^{v_i(x)})'_0$, $i=1,2$.  

Then the   maps $\Omega\ni x\mapsto |{\sf S}_x^1|_{u(x)},\sfd_{u(x)}({\sf S}_x^1,{\sf S}_x^2),\la {\sf S}_x^1,{\sf S}_x^2\ra_{u(x)}$ are Borel. 
\end{lemma}
\begin{proof}
It is sufficient to prove that $\Omega\ni x\mapsto  \sfd_{u(x)}({\sf S}_x^1,{\sf S}_x^2)\in\R$ is Borel, as then the other Borel regularities will follow. We have already noticed that the maps $x\mapsto (\G_u^{v^i})_{\alpha^it}(x)\in\Y$, $i=1,2$, are Borel, hence so is the map 
\[
x\mapsto\frac{\sfd_\Y\big((\G_u^{v^1})_{\alpha^1t}(x),(\G_u^{v^2})_{\alpha^2t}(x)\big)}{t}
\]
for any $0<t\ll1$. Since these maps pointwise converge to $ x\mapsto  \sfd_{u(x)}({\sf S}_x^1,{\sf S}_x^2)$ as $t\downarrow0$, the claim follows.
\end{proof}
In particular, for ${\sf S}^1,{\sf S}^2$ as in the above statement, the quantity 
 \begin{equation}
\label{eq:defdl2}
\sfd_{L^2}\big({\sf S}^1,{\sf S}^2\big):=\sqrt{\int_\Omega  \sfd^2_{u(x)}({\sf S}^1_x,{\sf S}^2_x)\,\d\mm(x)},
\end{equation}
is well defined. Standard arguments then show that $\sfd_{L^2}$ is symmetric, satisfies the triangle inequality and $\sfd_{L^2}({\sf S},{\sf S})=0$ (but it might happen that $\sfd_{L^2}({\sf S}^1,{\sf S}^2)=0$ for ${\sf S}^1\neq{\sf S}^2$ and that $\sfd_{L^2}({\sf S}^1,{\sf S}^2)=+\infty$).

We then give the following definitions:
\begin{definition}[$\mathcal L^2$ sections of $u^*\T_G\Y$]\label{def:l2se1}
Let $(\Y,\sfd_\Y,\bar y)$ be a pointed   \Cat0-space, $(\X,\sfd,\mm)$ a metric measure space, $\Omega\subset\X$ an open subset and $u$ a Borel representative of a map in $L^2(\Omega,\Y_{\bar y})$. Then, $\mathcal L^2(u^*\T_G\Y,\mm\restr\Omega)$ is the collection  of sections ${\sf S}$ of $u^*\T_G\Y$ such that:
\begin{itemize}
\item[i)] For any $\alpha\in\R^+$ and $v\colon \Omega\to\Y$ Borel and essentially separably valued  we have that $x\mapsto \sfd_{u(x)}({\sf S}_x,\alpha(\G_{u(x)}^{v(x)})'_0)$ is  a Borel function.
\item[ii)] There is a sequence $(\alpha_n)\subset\R^+$ and maps  $v_n\colon \Omega\to\Y$ Borel and essentially separably valued such that for the sections ${\sf S}^n$ given by ${\sf S}^n_x:=\alpha_n(\G_{u(x)}^{v_n(x)})'_0$ we have
\begin{equation}
\label{eq:hl2}
\begin{split}
\sup_{n\in\N}\sfd_{L^2}({\sf S}^n,{\sf 0})&<\infty,\\
\lim_{n\to\infty}\sfd_{u(x)}({\sf S}^n_x,{\sf S}_x)&= 0\qquad\forall x\in\Omega.
\end{split}
\end{equation}
\end{itemize}
\end{definition}
It is clear from the definitions that for ${\sf S}^1,{\sf S}^2\in \mathcal L^2(u^*\T_G\Y,\mm\restr\Omega)$ the map $x\mapsto \sfd_{u(x)}({\sf S}^1_x,{\sf S}^2_x)$ is Borel and $L^2(\mm\restr\Omega)$-integrable, therefore $\sfd_{L^2}({\sf S}^1,{\sf S}^2)$ is well defined by \eqref{eq:defdl2} and finite. 

\begin{definition}[$ L^2$ sections of $u^*\T_G\Y$]\label{def:l2se2}
Let $(\Y,\sfd_\Y,\bar y)$ be a pointed   \Cat0-space, $(\X,\sfd,\mm)$ a metric measure space, $\Omega\subset\X$ an open subset and $u$ a Borel representative of a map in $L^2(\Omega,\Y_{\bar y})$. We define $ L^2(u^*\T_G\Y,\mm\restr\Omega)$ as the quotient of $\mathcal L^2(u^*\T_G\Y,\mm\restr\Omega)$ with respect to the relation ${\sf S}^1\sim{\sf S}^2$ if $\sfd_{L^2}({\sf S}^1,{\sf S}^2)=0$.
\end{definition}
It is obvious that the relation indicated in the previous definition  is an equivalence relation, so that $ L^2(u^*\T_G\Y,\mm\restr\Omega)$ is well defined. Also, the quantity $\sfd_{L^2}$ passes to the quotient and defines a distance, still denoted by $\sfd_{L^2}$, on $ L^2(u^*\T_G\Y,\mm\restr\Omega)$ and standard considerations show that the resulting object is a complete metric space.

Now let $\tilde u\colon \Omega\to\Y$ be Borel and $\mm$-a.e.\ equal to $u$ and consider the identification $I\colon \mathcal  L^2(u^*\T_G\Y,\mm\restr\Omega)\to \mathcal L^2(\tilde u^*\T_G\Y,\mm\restr\Omega)$ sending ${\sf S}$ to the section $I({\sf S})$ defined by
\[
I({\sf S})_x:=\left\{\begin{array}{ll}
{\sf S}_x,&\qquad\text{ if }u(x)=\tilde u(x),\\
0_{\tilde u (x)},&\qquad\text{ if }u(x)\neq \tilde u(x).
\end{array}\right.
\]
It is clear that this map passes to the quotients and thus induces a map, still denoted by $I$, from $ L^2(u^*\T_G\Y,\mm\restr\Omega)$ to $L^2(\tilde u^*\T_G\Y,\mm\restr\Omega)$. Also, the fact that $u=\tilde u$ $\mm$-a.e.\ trivially implies  that such $I$ is an isometry.

Thanks to these considerations, it makes sense to consider the space $L^2(u^*\T_G\Y,\mm\restr\Omega)$ for $u\in L^2(\Omega,\Y_{\bar y})$, i.e.\ even when $u$ is only given up to $\mm$-a.e.\ equality: it is just sufficient to pick any Borel representative of $u$, consider the corresponding space of $L^2$-sections up to $\mm$-a.e.\ equality and notice that such space does not depend on the representative of $u$ chosen.

The basic properties of the space $L^2(u^*\T_G\Y,\mm\restr\Omega)$ are collected in the following statement.
\begin{proposition}[Properties of $L^2(u^*\T_G\Y,\mm\restr\Omega)$]\label{prop:propl2}
Let $(\Y,\sfd_\Y,\bar y)$ be a pointed   \Cat0-space, $(\X,\sfd,\mm)$ a metric measure space, $\Omega\subset\X$ an open subset and $u\in L^2(\Omega,\Y_{\bar y})$. 

Then:
\begin{itemize}
\item[$(i)$] For every ${\sf S}^1,{\sf S}^2\in L^2(u^*\T_G\Y,\mm\restr\Omega)$ the functions $\Omega\ni x\mapsto \sfd_{u(x)}({\sf S}^1_x,{\sf S}^2_x), |{\sf S}^1_x|_{u(x)},\la{\sf  S}^1_x,{\sf S}^2_x\ra_{u(x)}$ are (equivalence classes up to $\mm$-a.e.\ equality of) Borel functions.
\item[$(ii)$]  For every ${\sf S}^1,{\sf S}^2\in L^2(u^*\T_G\Y,\mm\restr\Omega)$ the section ${\sf S}^1\oplus{\sf S^2}$   given by  the (equivalence class of the) map $x\mapsto (x,u(x),{\sf S}^1_x\oplus{\sf S}^2_x)$ belongs to $L^2(u^*\T_G\Y,\mm\restr\Omega)$.
\item[$(iii)$] For every ${\sf S}\in L^2(u^*\T_G\Y,\mm\restr\Omega)$ and $f\in L^\infty(\Omega)$ the  section $f{\sf S}$ given by the (equivalence class of the) map $x\mapsto (x,u(x),f(x){\sf S}^1_x)$  belongs to $L^2(u^*\T_G\Y,\mm\restr\Omega)$.
\end{itemize}
\end{proposition}
\begin{proof} The Borel regularity of $ x\mapsto \sfd_{u(x)}({\sf S}^1_x,{\sf S}^2_x)$ has already been noticed. Then the one of $|{\sf S}^1_x|_{u(x)}$ follows from the fact that the zero section ${\sf 0}$ belongs to $L^2(u^*\T_G\Y,\mm\restr\Omega)$ and thus the one of $\la{\sf  S}^1_x,{\sf S}^2_x\ra_{u(x)}$ follows by the definition of scalar product.

We pass to $(ii)$ and first consider the case of  ${\sf S}^1,{\sf S}^2\in \mathcal L^2(u^*\T_G\Y,\mm\restr\Omega)$ of the form ${\sf S}^1_x=\alpha(\G_{u(x)}^{v(x)})'_0$ and ${\sf S}^2_x=\beta(\G_{u(x)}^{w(x)})'_0$ for $v,w\colon \Omega\to\Y$ Borel and essentially separably valued and $\alpha,\beta\geq 0$. Put $T:=\min\{\alpha^{-1},\beta^{-1}\}\in(0,\infty]$ and for $t\in(0,T)$ put $v_t:=(\G_u^v)_{\alpha t}$, $w_t:=(\G_u^w)_{\beta t}$ and let $m_t(x)$ be the midpoint of $v_t(x),w_t(x)$ for every $x\in\Omega$. From the continuity of the `midpoint' operation and the triangle inequality it easily follows that $x\mapsto m_t(x)$ is Borel, essentially separably valued and in $L^2(\Omega,\Y_{\bar y})$. Then, define the section ${\sf M}_t$ as  ${\sf M}_{t,x}:=\frac1t(\G_{u(x)}^{m_t(x)})'_0$ and  recall \eqref{eq:sumexpl} to see that ${\sf M}_{t,x}\to \frac12({\sf S}^1\oplus {\sf S}^2)_x$ in $\T_{u(x)}\Y$ as $t\downarrow0$ for every $x\in\Omega$: this proves that $\frac12({\sf S}^1\oplus {\sf S}^2)$ satisfies the requirement $(i)$ in Definition \ref{def:l2se1}. The same convergence together with the bound
\[
\begin{split}
|{\sf M}_{t,\cdot}|_{u(\cdot)}&\leq\tfrac1t\sfd_\Y(u,m_{t}) \leq \tfrac2t\big(\sfd_\Y(u,v_{t})+\sfd_\Y(u,w_{t}) \big)\leq 2\big(\alpha\sfd_\Y(u,v)+\beta\sfd_\Y(u,w) \big)\quad\text{on }\Omega
\end{split}
\]
valid for every $t\in(0,T)$ shows that $\frac12({\sf S}^1\oplus {\sf S}^2)$ satisfies also the requirement $(ii)$ in Definition \ref{def:l2se1}. 

Now the fact that $\frac12({\sf S}^1\oplus {\sf S}^2)\in \mathcal L^2(u^*\T_G\Y,\mm\restr\Omega)$  for generic ${\sf S}^1,{\sf S}^2\in \mathcal L^2(u^*\T_G\Y,\mm\restr\Omega)$ follows by approximation (recall point $(ii)$ in Definition \ref{def:l2se1}) and the continuity of the `sum' operation noticed in Proposition \ref{prop:hilbertine}, then the analogous properties for elements of $ L^2(u^*\T_G\Y,\mm\restr\Omega)$ trivially follow.

Finally, the fact that $\frac12({\sf S}^1\oplus {\sf S}^2)\in  L^2(u^*\T_G\Y,\mm\restr\Omega)$ implies ${\sf S}^1\oplus {\sf S}^2\in  L^2(u^*\T_G\Y,\mm\restr\Omega)$ is trivial from the definitions (see also the arguments below).

For $(iii)$  we notice that it is sufficient to prove that $f{\sf S}$ is in $\mathcal L^2(u^*\T_G\Y,\mm\restr\Omega)$ whenever   $f\colon \Omega\to\R$ is Borel and bounded and ${\sf S}\in\mathcal L^2(u^*\T_G\Y,\mm\restr\Omega)$. In this case  the   fact that $f{\sf S}$ satisfies the requirement $(i)$  in Definition \ref{def:l2se1} is obvious. For $(ii)$ we consider sections ${\sf S}^n_x= \alpha_n(\G_{u(x)}^{v_n(x)})'_0$ for which  \eqref{eq:hl2} hold and put $\tilde {\sf S}^n_x:=\alpha_n \|f\|_{L^\infty}(\G_{u(x)}^{w_n(x)})'_0$, where   $w_n(x):=(\G_{u(x)}^{v_n(x)})_{f(x)/\|f\|_{L^\infty}}$. The fact that the $w_n$'s are Borel representatives of maps in $L^2(\Omega,\Y_{\bar y})$ can be easily checked from the definition while  the fact that \eqref{eq:hl2} holds for $f{\sf S}$ and $(\tilde{\sf S}^n)$ is obvious. 
\end{proof}

Let us now come back to the initial discussion and, for given $u\in L^2(\Omega,\Y_{\bar y})$, let us define the map $\iota\colon \Geo_uL^2(\Omega,\Y_{\bar y})\to L^2(u^*\T_G\Y,\mm\restr\Omega)$ as follows. For  $v\in L^2(\Omega,\Y_{\bar y})$ and  $\alpha\geq 0$ we send the geodesic $t\mapsto (\G_u^v)_{\alpha t}$  to the (equivalence class of the) section given by $x\mapsto \alpha(\G^{v(x)}_{u(x)})'_{0}$. The relation between $\T_uL^2(\Omega,\Y_{\bar y})$ and $L^2(u^*\T_G\Y,\mm\restr\Omega)$ is then described by the following result:
\begin{proposition}[$L^2(u^*\T_G\Y,\mm\restr\Omega)$ and $\T_uL^2(\Omega,\Y_{\bar y})$]\label{prop:l2gen}
Let $(\Y,\sfd_\Y,\bar y)$ be a pointed   \Cat0-space, $(\X,\sfd,\mm)$ a metric measure space, $\Omega\subset\X$ an open subset and $u\in L^2(\Omega,\Y_{\bar y})$. 

Then, the map $\iota\colon \Geo_uL^2(\Omega,\Y_{\bar y})\to L^2(u^*\T_G\Y,\mm\restr\Omega)$ passes to the quotient and induces a map, still denoted $\iota$, from $\Geo_uL^2(\Omega,\Y_{\bar y})/\sim$ to  $L^2(u^*\T_G\Y,\mm\restr\Omega)$ that can be uniquely extended by continuity to a bijective isometry, again denoted $\iota$, from $\T_uL^2(\Omega,\Y_{\bar y})$ to $L^2(u^*\T_G\Y,\mm\restr\Omega)$.

Moreover, the so defined  isometry $\iota$  respects the operations on the tangent cones, i.e.
\begin{subequations}
\begin{align}
\label{eq:e1}
|{\sf v}|_u^2&=\int_\Omega |\iota({\sf v})_x|^2_{u(x)}\,\d\mm(x),\\
\label{eq:e2}
\la {\sf v}_1,{\sf v}_2\ra_u&=\int_\Omega \la \iota({\sf v}_1)_x,\iota({\sf v}_2)_x\ra_{u(x)}\,\d\mm(x),\\
\label{eq:e3}
\sfd^2_u( {\sf v}_1,{\sf v}_2)&=\int_\Omega \sfd^2_{u(x)}( \iota({\sf v}_1)_x,\iota({\sf v}_2)_x)\,\d\mm(x),\\
\label{eq:e4}
\iota(\lambda {\sf v})&=\lambda\iota({\sf v}),\\
\label{eq:e5}
\iota({\sf v}_1\oplus {\sf v}_2)&=\iota({\sf v}_1)\oplus\iota({\sf v}_2),
\end{align}
\end{subequations}
for any ${\sf v},{\sf v}_1,{\sf v}_2\in \T_uL^2(\Omega,\Y_{\bar y})$ and $\lambda\in\R^+$.
\end{proposition}
\begin{proof} Let $v^1,v^2\in L^2(\Omega,\Y_{\bar y})$, $\alpha_1,\alpha_2\geq 0$, consider the sections ${\sf S}^1,{\sf S}^2\in L^2(u^*\T_G\Y,\mm\restr\Omega)$ given by ${\sf S}^i:=\iota(\alpha_i(\G_{u}^{v^i})'_0)$. Notice that
\[
\begin{split}
\sfd^2_u\big( \alpha_1(\G_u^{v_1})'_0,\alpha_2(\G_u^{v_2})'_0\big)&=\lim_{t\downarrow0}\frac{\sfd_{L^2}^2\big((\G_u^{v_1})_{\alpha_1t},(\G_u^{v_2})_{\alpha_2t}\big)}{t^2}\\
&=\lim_{t\downarrow0}\int_\Omega\frac{\sfd^2_\Y\big((\G_{u(x)}^{v_1(x)})_{\alpha_1t},(\G_{u(x)}^{v_2(x)})_{\alpha_2t}\big)}{t^2}\,\d\mm(x)\\
&=\int_\Omega\lim_{t\downarrow0}\frac{\sfd^2_\Y\big((\G_{u(x)}^{v_1(x)})_{\alpha_1t},(\G_{u(x)}^{v_2(x)})_{\alpha_2t}\big)}{t^2}\,\d\mm(x)\\
&=\int_\Omega \sfd^2_{u(x)}\big({\sf S}^1_x,{\sf S}^2_x\big)\,\d\mm(x),
\end{split}
\]
where, in passing the limit inside the integral, we used the dominate convergence theorem and the fact that the integrand is non-negative and non-decreasing in $t$ (recall \eqref{eq:mondis}). This proves at once that $\iota$ passes to the quotient to a map on $\Geo_uL^2(\Omega,\Y_{\bar y})/\sim$ and that the so induced map is an isometry which therefore can be extended to a map from $\T_uL^2(\Omega,\Y_{\bar y})$ to $L^2(u^*\T_G\Y,\mm\restr\Omega)$. The fact that such extension is surjective follows from an approximation argument based on the requirement $(ii)$ in Definition \ref{def:l2se1}.


Now observe that  \eqref{eq:e3} has already been proved by the fact that $\iota$ is an isometry. Then \eqref{eq:e1} and \eqref{eq:e2} follow as well. Also, \eqref{eq:e4} is obvious by definition and then \eqref{eq:e5} follows from \eqref{eq:e3}, \eqref{eq:e4} and the metric characterization of the  midpoints of $x,y$ as the point $m$ such that $\sfd^2(x,m)+\sfd^2(m,y)=\sfd^2(x,y)/2$.
\end{proof}


\subsubsection{The separable setting}

In this section we  assume instead that $\Y$ is a separable and locally $\Cat\kappa$-space and we study the Borel structure of the pullback $u^*\T_G\Y$ of the geometric tangent bundle of $\Y$. We shall then see in the space case of $\Y$ being separable and $\Cat0$ how such Borel structure relates to the space $L^2(u^*\T_G\Y,\mm\restr\Omega)$ studied in the previous section.

Thus let  $\Y$ be separable and locally $\Cat\kappa$, $(\X,\sfd)$ be a complete and separable metric space and $\Omega\subset\X$ be open.

As before, for a given Borel map $u\colon \Omega\to\Y$,  the pullback geometric tangent bundle $u^*\T_G\Y$ is defined as 
\[
u^*\T_G\Y:=\big\{(x,y,v)\colon   x\in\Omega,\ y=u(x),\ v\in\T_y\Y\big\}\subset \X\times\T_G\Y
\]
and a section  of $u^*\T_G\Y$ is a map $\sfS \colon  \Omega \rightarrow u^*\T_G\Y$ such that $\sfS_x \in \T_{u(x)}\Y$ for every $x \in \Omega$.

Now equip $u^*\T_G\Y\subset \X\times\T_G\Y$  with the restriction of the product $\sigma$-algebra $\cB(\X)\otimes \cB(\T_G\Y)$, which, with abuse of terminology, we shall call Borel $\sigma$-algebra on $u^*\T_G\Y$ and denote $\cB(u^*\T_G\Y)$.   In particular, we shall say that a section is Borel if it is measurable w.r.t.\ $\cB(\X)$ and $\cB(u^*\T_G\Y)$.

A section is \emph{simple} provided there are a Borel partition $(E_n)$ of $\Omega$, $(\alpha_n)\subset \R^+$ and points $(y_n)\subset\Y$ s.t. $y_n \in B_{r_{u(x)}}(u(x))$, for every $x \in E_n$ and  $\sfS\restr{E_n}=\alpha_n(\G_{u(\cdot)}^{y_n})'_0$. We shall formally denote such section by  $\sum_n\nchi_{E_n}\alpha_nu^*(\G_{\cdot}^{y_n})'_0$. Notice that the restriction of such section to $E_n$ coincides with the (graph of) the composition of $u$ with the simple section of $\T_G\Y$ given by $y\mapsto(y,\alpha_n(\G_{y}^{y_n})'_0)$. In particular, recalling Proposition \ref{prop:sb} we see that simple sections of $u^*\T_G\Y$ are Borel.

Moreover, they are dense in the space of Borel sections:
\begin{lemma}[Density of simple sections]\label{le:denssimp} 
Let $(\X,\sfd)$ be a metric space, $(\Y,\sfd_\Y)$ be separable and locally $\Cat\kappa$-space,  $\Omega\subset\X$ an open subset and $u\colon \Omega\to\Y$ Borel. Let ${\sf S}\colon \Omega\to u^*\T_G\Y$ be a Borel section of $u^*\T_G\Y$ and $\eps>0$. 

Then, there is a simple section ${\sf T}$ such that $\sfd_{u(x)}({\sf S}_x,{\sf T}_x)<\eps$ for every $x\in\Omega$.
\end{lemma}
\begin{proof}
We can reduce the proof to the case of $\Y$ being $\Cat\kappa$ by using the Lindel\"of property of $\Y$ and the coverings made by $B_{\sfr_y/2}(y)$. Doing so, we achieve uniqueness of geodesics between any couple of points. Let $D\subset\Y$ be countable and dense $(y_n,\alpha_n)$ be an enumeration of $D\times\Q^+$. Then  for every $n\in\N$ consider the function $F_{n}\colon \T_G\Y\to\R$ given by  
\[
F_{n}(y,v):=\sfd_{y}(v,\alpha_n(\G_y^{y_n})'_0)=\sqrt{|v|_y^2+|\alpha_n|^2\sfd^2(y,y_n)-2\la v,\alpha_n(\G_y^{y_n})'_0\ra_y}.
\]
The defining requirements of $\cB(\T_G\Y)$ and the property \eqref{eq:normbor} ensure that $F_{n}$ is Borel. Hence so is the map $\tilde F_n\colon u^*\T_G\Y\to\R$ defined as $\tilde F_n:=F_n\circ \pi_{\T_G\Y}$,  where $\pi_{\T_G\Y}\colon u^*\T_G\Y\subset\X\times \T_G\Y\to \T_G\Y$ is the canonical projection.

Hence given a Borel section ${\sf S}$ of $u^*\T_G\Y$ the map $\tilde F_n\circ {\sf S}\colon \X\to\R$ is Borel and thus, for given $\eps>0$, so is the set $\tilde E_n:=(\tilde F_n\circ {\sf S})^{-1}([0,\eps))$. We then put $E_n:=\tilde E_n\setminus\cup_{i<n}\tilde E_i$ and notice that the property \eqref{eq:densecone} ensures that the $E_n$'s form a partition of $\X$, thus giving the conclusion.
\end{proof}
Thanks to such density result we can show that the operations on the tangent cones preserve Borel regularity. The statement below is similar in spirit to (part of) the statement of Proposition \ref{prop:propl2}, but here no measure is fixed on $\Omega$ and that the sections are defined for every $x\in\Omega$, not for $\mm$-a.e.\ $x$.
\begin{proposition} Let $(\X,\sfd)$ be a metric space, $(\Y,\sfd_\Y)$ be separable and locally $\Cat\kappa$,  $\Omega\subset\X$ an open subset and $u\colon \Omega\to\Y$ a Borel map. Let ${\sf S}^1,{\sf S}^2$ be Borel sections of $u^*\T_G\Y$ and $f\colon \X\to\R^+$ be a Borel map. 

Then the functions sending $x\in\X$ to $|{\sf S}^2_x|_{u(x)},\la {\sf S}^1_x,{\sf S}^2_x\ra_{u(x)},\sfd_{u(x)}({\sf S}^1_x,{\sf S}^2_x)$ are Borel and the sections $x\mapsto f(x){\sf S}^1_x,{\sf S}^1_x\oplus {\sf S}^2_x$ are Borel as well.
\end{proposition}
\begin{proof} Let ${\sf S}^1,{\sf S}^2$ be simple of the form: ${\sf S}^1=\nchi_{E_1}\alpha \, u^*(\G_{\cdot}^{y_1})'_0$ and ${\sf S}^2=\nchi_{E_2}\beta\, u^*(\G_{\cdot}^{y_2})'_0$, with $E_i:=u^{-1}(A_i)$ and $A_i \in \cB(\Y)$ such that $y_i \in B_{\sfr_y}(y)$ for every $y \in A_i,i=1,2$. Then they are the (graph of the) composition of $u$ with the simple sections of $\T_G\Y$ given by $\nchi_{A_1}\alpha \,(\G_{\cdot}^{y_1})'_0$ and $\nchi_{A_2}\beta\, (\G_{\cdot}^{y_2})'_0$ respectively, hence in this case the conclusion comes from Proposition \ref{prop:bormap}.

Then the conclusion comes from the `fiberwise' continuity of all the expressions considered (granted by Proposition \ref{prop:hilbertine}) and the density of simple sections established in Lemma \ref{le:denssimp} above.
\end{proof}

We now come to the relation between the space of (equivalence classes up to $\mm$-a.e.\ equality of) Borel sections of $u^*\T_G\Y$ and the space $L^2(u^*\T_G\Y,\mm\restr\Omega)$ in the case where $\Y$ is separable and $\Cat0$. As expected, these spaces coincide when the right integrability of the first ones is in place:
\begin{proposition}\label{prop:link} 
Let $(\X,\sfd,\mm)$ be a metric measure space, $(\Y,\sfd_\Y,\bar y)$ be a pointed separable \Cat0-space,  $\Omega\subset\X$ an open subset and $u\colon \Omega\to\Y$ be a Borel map.

Then, ${\sf S}\in L^2(u^*\T_G\Y,\mm\restr\Omega)$  if and only if it is the equivalence class up to $\mm$-a.e.\ equality of a Borel section ${\sf T}$ of $u^*\T_G\Y$ with $\int_\Omega|{\sf T}|^2_{u(x)}\,\d\mm(x)<\infty$.
\end{proposition}
\begin{proof}
Assume at first that ${\sf S}\in L^2(u^*\T_G\Y,\mm\restr\Omega)$. Then the fact that $\int_\Omega|{\sf S}_x|^2_{u(x)}\,\d\mm(x)<\infty$ is a direct consequence of the definition and of Proposition \ref{prop:l2gen} above, thus we only need to prove that ${\sf S}$ is the equivalence class up to $\mm$-a.e.\ equality of a Borel section of $u^*\T_G\Y$. To see this we need to prove that, letting $\pi_\X,\pi_{\T_G\Y}$ be the projections of $u^*\T_G\Y\subset \X\times \T_G\Y$ to $\X,\T_G\Y$ respectively, the maps $\pi_\X\circ{\sf S}$ and $\pi_{\T_G\Y}\circ{\sf S}$ are equivalence classes up to $\mm$-a.e.\ equality of Borel maps. For the first one this is obvious, because it is the identity on $\X$. For the second one we recall the definition of $\cB(\T_G\Y)$ to see that we need to prove that $\pi_\Y\circ\pi_{\T_G\Y}\circ{\sf S}$ is Borel (which it is, because it coincides with $u$) and that $x\mapsto \la {\sf S}_x,(\G_{u(x)}^z)'_0\ra$ is Borel for every $z\in\Y$ (which is easily seen to be the case from the requirement $(i)$ in Definition \ref{def:l2se1}).

We pass to the converse implication and start observing that Lemma \ref{le:l2s} and the definition of  $\cB(\T_G\Y)$  just recalled ensure that for any $v\in L^2(\Omega,\Y_{\bar y})$ the section given by $(\G_{u(x)}^{v(x)})'_0$ is the equivalence class up to $\mm$-a.e.\ equality of a Borel section. It follows that if  ${\sf T}$ is a Borel section as in the statement, then it satisfies the requirement $(i)$ in Definition \ref{def:l2se1}. We now claim that if ${\sf T}$ is also simple, then it also satisfies the requirement $(ii)$. To see this write ${\sf T}=  \sum_n\nchi_{E_n}\alpha_n\,u^*(\G_{\cdot}^{y_n})'_0$ and put ${\sf T}^i:=\sum_{n\leq i}\nchi_{E_n}\alpha_n\,u^*(\G_{\cdot}^{y_n})'_0$ where it is intended that for $x\notin \cup_{n\leq i}E_n$ we have ${\sf T}^i_x=0_{u(x)}\in\T_{u(x)}\Y$. Then putting $\beta_i:=\max_{n\leq i}\alpha_n$, $y_{n,i}:=(\G_{u(x)}^{y_n})_{\alpha_n/\beta_i}$ and defining $v^i\in L^2(\Omega,\Y_{\bar y})$ as $v^i\restr{E_n}:=y_{n,i}$ for $n\leq i$ and $v^i\restr{\Omega\setminus\cup_{n\leq i}E_n}\equiv u$ we see that ${\sf T}^i=\iota(\beta_i(\G_u^{v^i})'_0)$, so that (the equivalence class up to $\mm$-a.e.\ equality of) ${\sf T}^i$ belongs to $L^2(u^*\T_G\Y,\mm\restr\Omega)$. It is then clear that $\sfd_{L^2}({\sf T}^i,{\sf T})\to0$, proving that the equivalence class of ${\sf T}$ belongs to $L^2(u^*\T_G\Y,\mm\restr\Omega)$.

Then the conclusion for a generic section ${\sf T}$  as in the statement can be easily obtained by an approximation argument starting from the density result in Lemma \ref{le:denssimp}.
\end{proof}

\subsection{The Korevaar-Schoen energy} 
We recall here the key definitions and results of \cite{GT20}, where the original analysis done in \cite{KS93} has been generalized to the setting of $\RCD(K,N)$ spaces (\cite{AmbrosioGigliSavare11-2}, \cite{Gigli12}).

For the definitions of all the objects appearing below we refer to \cite{GT20} (but see also \cite{GPS18} for the definition of the differential $\d u$ appearing in the statement below).
\begin{theorem}[The Korevaar-Schoen energy]\label{thm:defks}
Let $(\X,\sfd,\mm)$ be a $\RCD(K,N)$ space, $K\in\R$, $N\in[1,\infty)$, $(\Y,\sfd_\Y,\bar y)$ a pointed \Cat0-space, $\Omega\subset\X$ open and $u\in L^2(\Omega,\Y_{\bar y})$. Then the following are equivalent:
\begin{itemize}
\item[i)] Letting $\sfks_{2,r}[u,\Omega]\colon \Omega\to\R^+$ be defined by
$$ \sfks_{2,r}[u,\Omega] (x) := \begin{cases} \ \Big\vert \fint_{B_r(x)} \frac{\sfd_{\Y}^2(u(x),u(\tilde{x}))}{r^2} \, \d \mm(\tilde{x}) \Big\vert^{1/2} &\text{ if } B_r(x) \subset \Omega, \\ \ 0 &\text{ otherwise.} \end{cases}$$
and the energy $\E^\sfKS(u)$ be given by
\begin{equation}
\label{eq:defE}
\E^\sfKS(u):= \limsup_{r\downarrow 0} \frac12\int_\Omega \sfks^2_{2,r}[u,\Omega]  \, \d\mm,
\end{equation}
we have $\E^\sfKS(u)<\infty$.
	\item[ii)] There is $G\in L^2(\Omega)$ such that for every $\varphi\colon \Y\to\R$ 1-Lipschitz with $\varphi(\bar y)=0$ we have $\varphi\circ u\in W^{1,2}(\Omega)$ with $|\d (\varphi\circ u)|\leq G$ $\mm$-a.e..
\end{itemize}
If any of these hold, the `energies at scale $r$'  $\sfks_{2,r}[u,\Omega]$ converge to $(d+2)^{-\frac12}|\d u|_{\sf HS}$ in $L^2(\Omega)$ as $r\downarrow0$. In particular, the $\lims$ in \eqref{eq:defE} is actually a limit and the energy admits the representation 
\[
\E^\sfKS(u)=\frac1{2(d+2)}\int_\Omega|\d u|_{\sf HS}^2\,\d\mm.
\]
Finally, the functional $\E^\sfKS\colon L^2(\Omega,\Y_{\bar y})\to[0,+\infty]$ is convex and lower semicontinuous.
\end{theorem}
\begin{remark}{\rm
It should be noticed that the smallest function $G$ for which $(ii)$ holds is not the Hilbert-Schmidt norm $|\d u|_{\sf HS}$ of the differential $\d u$ of $u$, but rather the (pointwise) operator norm of $\d u$. The two quantities are nevertheless comparable, i.e.\ one controls the other up to multiplication with a dimensional constant. 
}\fr\end{remark}
We shall denote by $\sfKS^{1,2}(\Omega,\Y_{\bar{y}})\subset L^2(\Omega,\Y_{\bar y})$ the collection of maps with finite energy and recall from \cite{GT20} that for $ u,v\in \sfKS^{1,2}(\Omega,\Y_{\bar{y}})$ we always have $\sfd_\Y(u,v)\in W^{1,2}(\Omega)$. Therefore it makes sense to ask whether $u,v$ attain the same boundary value by checking whether or not we have $\sfd_\Y(u,v)\in W^{1,2}_0(\Omega)$. 

Then given $\bar u\in \sfKS^{1,2}(\Omega,\Y_{\bar{y}})$ the `energy $\E^\sfKS_{\bar u}\colon L^2(\Omega,\Y)\to [0,\infty]$ with $\bar u$ as prescribed boundary value' can be defined as
\[
\E^\sfKS_{\bar u}(u):=\left\{\begin{array}{ll}
\E^\sfKS(u)&\qquad\text{if $u\in \sfKS^{1,2}(\Omega,\Y_{\bar{y}})$ and $\sfd_\Y(u,\bar u)\in W^{1,2}_0(\Omega)$},\\
+\infty&\qquad\text{otherwise}.
\end{array}
\right.
\] 
We shall denote the domain of $\E^\sfKS_{\bar u}$ by $\sfKS^{1,2}_{\bar u}(\Omega,\Y_{\bar{y}})\subset L^2(\Omega,\Y_{\bar y})$ and recall that 
\begin{equation}
\label{eq:proprE}
\E^{\sfKS}_{\bar u}\colon L^2(\Omega,\Y_{\bar y})\to[0,+\infty]\qquad\text{is convex and lower semicontinuous,}
\end{equation}
moreover it admits a unique minimizer,  called harmonic map with $\bar u$ as boundary value.

\begin{remark}{\rm
Even if the definition of $\E^\sfKS_{\bar u}$ can be given in high generality, it should be noted that it may happen that $\E^\sfKS_{\bar u}=\E^\sfKS$. This happens when $W^{1,2}_0(\Omega)=W^{1,2}(\Omega)$ which in turn occurs if $\X\setminus\Omega$ has null capacity. Thus in practical situations if one wants to enforce some boundary condition, it should be checked that  $\X\setminus\Omega$ has positive capacity.
}\fr\end{remark}

For later use we recall that the convexity of both $\E^{\sfKS}$ and $\E^{\sfKS}_{\bar u}$ can  be improved to the following inequality:
\begin{equation}
\label{eq:KSint}
\E^{\sfKS}((\G_u^v)_t)+t(1-t)\E^{\sfKS}(d)\leq (1-t)\E^{\sfKS}(u)+t\E^{\sfKS}(v)\qquad\forall t\in[0,1],
\end{equation}
where $d(x):=\sfd(u(x),v(x))$. Such inequality has been proved for the case $t=\frac12$ in \cite{GT20} (imitating the arguments in \cite{KS93}), the general case follows along the same arguments. It is worth to underline that in the above the maps $u,v,(\G_u^v)_t$ are $\Y$-valued, while $d$ is real valued. In this sense the energy of $\E^{\sfKS}(d)$ of $d$ has a different meaning w.r.t.\ the energy of the other maps. Still,  we recall (see \cite{GT20} and \cite{KS93}) that for a constant $c(d)$ depending only on the essential dimension $d\le N$ of $\X$ we have $\E^{\sfKS}(f)=c(d){\sf Ch}(f)$ for any $f\in L^2(\Omega)$, where ${\sf Ch}$ is the standard Cheeger/Dirichlet energy on $\X$.

\subsection{The Laplacian of a \Cat0-valued map}

Let us start by giving the general definition of Laplacian of a \Cat0-valued Sobolev map:
\begin{definition}[Tension field/Laplacian]\label{def:laplacian}
Let $(\X,\sfd,\mm)$ be a $\RCD(K,N)$ space, $\Omega\subset\X$ an open subset, $(\Y,\sfd_\Y,\bar y)$ a pointed \Cat0-space and $\bar u\in \sfKS^{1,2}(\Omega,\Y_{\bar{y}})$.

Then the domain of the Laplacian $D(\Delta_{\bar u})\subset  \sfKS_{\bar u }^{1,2}(\Omega,\Y_{\bar{y}})$  is defined as $D(\Delta_{\bar u}):=D(|\partial^-\E^\sfKS_{\bar u}|)$ and for $u\in D(\Delta_{\bar u})$ we put 
\[
\Delta_{\bar u} u:=\iota({\sf v})\in L^2(u^*\T_G\Y,\mm\restr\Omega) , \quad\text{ where ${\sf v}$ is the element of minimal norm in $-\partial^-\E^\sfKS_{\bar u}(u)$.}
\]
Similarly, for maps $u$ from $\X$ to $\Y$ we say that $u$ is in the domain of the Laplacian $D(\Delta)$ if $|\partial^-\E^\sfKS|(u)<\infty$ and in this case $\Delta u:=\iota({\sf v})$, where $\iota({\sf v})$ is  the element of minimal norm in $-\partial^-\E^\sfKS(u)$.
\end{definition}

\begin{proposition}[Laplacian and variation of the energy]\label{prop:varen}
Let $(\X,\sfd,\mm)$ be a $\RCD(K,N)$ space, $\Omega\subset\X$ an open subset, $(\Y,\sfd_\Y,\bar y)$ a pointed \Cat0-space and $\bar u\in \sfKS^{1,2}(\Omega,\Y_{\bar{y}})$. Also, let $u\in D(\Delta_{\bar u})$. Then, for every $v\in L^2(\Omega,\Y_{\bar y})$, we have
\begin{equation}
\label{eq:varen}
-\int_{\X}\la \Delta_{\bar u} u(x), \big(\G_{u(x)}^{v(x)}\big)'_0\ra_{u(x)} \  \d \mm(x)\leq \lim_{t\downarrow0}\frac{\E_{\bar u}^{\sfKS}((\G_u^v)_t)-\E_{\bar u}^{\sfKS}(u)}t.
\end{equation}
Moreover, $u$ is harmonic with $\bar u $ as boundary value if and only if $u\in D(\Delta_{\bar u})$ with $\Delta_{\bar u} u=0$.
\end{proposition}
\begin{proof}
Inequality \eqref{eq:varen} follows applying  \eqref{eq:subder}, the definition of $\Delta_{\bar u} u$ and recalling Proposition \ref{prop:l2gen}. The second claim is a restatement of \eqref{eq:sub0} in this setting.
\end{proof}
\begin{remark}{\rm
This last proposition shows that our definition is compatible with the classical one valid in the smooth category. Indeed, if $\X,\Y$ are smooth Riemannian manifold, $\bar u,u\colon \bar \Omega\subset\X\to\Y$ are smooth maps with the same boundary values, ${\sf v}$ is a smooth section of $u^*\T\Y$ (in the smooth setting $\T_G\Y$ is canonically equivalent to the standard tangent bundle $\T\Y$) which is 0 on $\partial\Omega$,  then we can produce a smooth perturbation of $u$ by putting $u_t(x):=\exp_{u(x)}(t{\sf v}_x)$. A direct computation then shows that
\[
\frac{\d}{\d t}\restr{t=0}\E^\sfKS_{\bar u}(u_t)=-\int_\Omega\la \tau(u)_x,{\sf v}_x\ra_{u(x)}\,\d\mm(x),
\]
where $\tau(u)$ is the \emph{tension field of $u$}, see for instance \cite[Section 9.2]{Jost17}. This formula is the smooth version of \eqref{eq:varen}. Notice indeed that $u_t=(\G_u^{u_1})_t$ for $t\in[0,1]$ (and similarly $u_t=(\G_u^{u_{-1}})_{-t}$ for $t\in[-1,0]$) and that if everything is smooth, then $t\mapsto \E^\sfKS_{\bar u}(u_t)$ is $C^1$, hence differentiable in 0, so that the one-sided bound in \eqref{eq:varen} becomes an equality in the smooth case.

It is worth to underline that in our framework the lack of equality in \eqref{eq:varen} is not only related to the lack of smoothness of $t\mapsto \E^\sfKS_{\bar u}(u_t)$, which a priori could produce different left and right derivatives in 0, but also to the fact that tangent \emph{cones} are not really tangent \emph{spaces}: the opposite of a vector field does not necessarily exist and thus we are forced to take one-sided perturbations only.
}\fr\end{remark}

A direct consequence of Proposition \ref{prop:varen} above is the following:
\begin{corollary}
With the same assumptions and notation as in Proposition \ref{prop:varen} we have
\[
\E_{\bar u}^{\sfKS}(u)-\int_{\X}\la \Delta_{\bar u} u(x), \big(\G_{u(x)}^{v(x)}\big)'_0\ra  _{u(x)} \  \d \mm(x)+\E^{\sfKS}(d)\leq \E_{\bar u}^{\sfKS}(v),
\]
where $d:=\sfd(u,v)$.
\end{corollary}
\begin{proof}
Couple \eqref{eq:varen} with \eqref{eq:KSint}.
\end{proof}

In the next discussion, we are interested in properties of the composition $f \circ u$, whenever $u$ is a harmonic map and $f$ is $\lambda$-convex functional. Observe that, in a smooth framework, the chain rule $\Delta (f\circ u) =$ Hess$ f(\nabla u,\nabla u)  + \d f(\Delta u)  $ immediately implies that
\begin{equation}
\label{eq:casoliscio}
\Delta (f\circ u)\geq \lambda|\d u|^2_{\sf HS}\qquad\text{if $f$ is $\lambda$-convex and $u$ is harmonic.}
\end{equation}
A nonsmooth version of \eqref{eq:casoliscio} has already been addressed in \cite{LS19} (see Theorem 1.2 there) for maps with \emph{euclidean} source domain and \Cat0-target. Nevertheless, as we are going to show in Theorem \ref{thm:nablafu}, the discussion generalizes to our framework: the main stumbling block to overcome being the absence of \emph{Lipschitz} vector field on a $\RCD$-space. In the next, we shall need the following property of Sobolev functions and, specifically, of their directional derivatives (for the definition of test vector field see \cite{Gigli14} and for the concept of Regular Lagrangian Flow see \cite{Ambrosio-Trevisan14}):
\begin{proposition}\label{prop:dercurv}
Let $(\X,\sfd,\mm)$ be a $\RCD(K,N)$ space, $(\Y,\sfd_\Y,\bar y)$ a pointed complete metric space, $\Omega\subset \X$ open, $v\in L^2_\mm(T\X)$ a test vector field and $(\sfFL_{s}^v)$ the associated Regular Lagrangian Flow. Also, let $u\in  \sfKS^{1,2}(\Omega,\Y_{\bar{y}})$.

Then, for every $K\subset\Omega$ compact, we have that 
\begin{equation}
\label{eq:dery}
\lim_{s\to 0}\frac{\sfd_\Y(u\circ\sfFL_{s}^v,u)}{s}=|\d u(v)|\qquad\text{in $L^2(K)$}.
\end{equation}
(notice that for $|s|$ small the map $u\circ\sfFL_{s}^v$ is well defined from $K$ to $Y$). 

Similarly, for a real valued Sobolev function $g\in W^{1,2}(\Omega)$ we have
\begin{equation}
\label{eq:derr}
\lim_{s\to 0}\frac{g\circ\sfFL_{s}^v-g}{s}=\d g(v)\qquad\text{in $L^2(K)$}.
\end{equation}
\end{proposition}
\begin{proof} Property \eqref{eq:derr} is (an equivalent version of) the definition of Regular Lagrangian Flow, see for instance \cite[Proposition 2.7]{GR17}. For \eqref{eq:dery} recall first \cite[Remark 4.15]{GT20} to get that functions in $\sfKS^{1,2}(\Omega,\Y_{\bar{y}})$ also belong to the `direction' Korevaar-Schoen space as defined in \cite{GT18}, then recall  \cite[Theorem 4.5]{GT18}.
\end{proof}
The next Lemma deals with variations of a map $u$, suitably obtained through gradient flows trajectories in the target space, and the rate of change at the level of Korevaar-Schoen energy (see \eqref{eq:stimapunt}-\eqref{eq:dervaru} below). In the following statement, notice that $f\circ u$ belongs to $W^{1,2}(\Omega)$ - and thus $\d(f\circ u)$ is well defined - because $f$ is Lipschitz, $\Omega$ has finite measure and by $(ii)$ in Theorem \ref{thm:defks}. Also, for the very same reason, we shall drop the subscript $\bar{y}$ from $\Y$ when $\Omega$ is bounded as the $L^2$-integrability depends no more on the particular chosen point $\bar{y} \in \Y$. Compare the proof with \cite[Lemma 3.1]{LS19}.
\begin{lemma}\label{lem:varEfu} Let $(\X,\sfd,\mm)$ be a $\RCD(K,N)$ space, $\Y$ \Cat0-space and $\Omega\subset \X$ open and bounded. Also, let   $f \in \Lip(\Y)$ be $\lambda$-convex, $\lambda\in\R$, and $u \in \sfKS^{1,2}(\Omega,\Y)$. For $g \in \Lip_{bs}(\X)^+$, define the (equivalence class of the) variation map
\[ 
u_t(x)={\sf GF}^f_{tg(x)}(u(x))\qquad\forall t>0,\ x\in\Omega.
\]

Then,  $u_t \in \sfKS^{1,2}(\Omega,\Y)$ for every $t>0$ and there is a constant $C>0$ depending on $f,g$ such that
\begin{equation}
\label{eq:stimapunt}
|\d u_t|_{\sf HS}^2\leq e^{-2\lambda tg }\big(|\d u|_{\sf HS}^2-2t \,\la \d g,\d(f\circ u)\ra+Ct^2\big)\qquad\mm\text{-}a.e.\ in \ \Omega,
\end{equation}
holds for every $t\in[0,1]$. In particular
\begin{equation}
\label{eq:dervaru}
\limsup_{t\downarrow 0} \frac{\E^{\sfKS}(u_t)-\E^{\sfKS}(u)}{t} \le -\int_{\Omega} \frac{\lambda}{d+2} g|\d u|^2_{\sf HS} + \la \d(f\circ u),\d g\ra \, \d \mm.
\end{equation}
\end{lemma}
\begin{proof} The map $x\mapsto(tg(x),u(x))$ is Borel and essentially separably valued and the map $(t,y)\mapsto{\sf GF}^f_t(y)$ is continuous, hence $x\mapsto u_t(x)$ is Borel and essentially separably valued. Also, the identity  \eqref{eq:metconv} and the trivial estimate $|\partial^-f|\leq\Lip(f)$ show that $t\mapsto {\sf GF}^f_t(y)$ is $\Lip(f)$-Lipschitz for every $y\in\Y$, thus $\sfd_\Y(u_t(x),\bar y)\leq t\sup(g)\Lip(f)+\sfd_\Y(u(x),\bar y)$, for every $\bar y \in \Y$, from which it easy follows that $u_t\in L^2(\Omega,\Y)$. Taking also into account the contraction property \eqref{eq:contr} we obtain  that
\[
\begin{split}
\sfd_\Y(u_t(x),u_t(y))&\leq e^{\lambda^- t(g(x)+g(y))}\sfd_\Y\big(u(x),{\sf GF}^f_{t|g(y)-g(x)|}(u(y))\big)\\
&\leq  e^{2\lambda^- t\sup g}\big(\sfd_\Y(u(x),u(y))+t\Lip(g)\Lip(f)\sfd(x,y)\big) 
\end{split}
\]
and thus
\[
\sfks_{2,r}^2[u_t,\Omega] (x)\leq 2e^{4\lambda^- t\sup g}\big(\sfks_{2,r}^2[u,\Omega] (x)+t^2\Lip^2(g)\Lip^2(f)\big).
\]
Integrating and using the fact that $\mm(\Omega)<\infty$ we conclude that $u_t\in \sfKS^{1,2}(\Omega,\Y)$.

In order to obtain  \eqref{eq:stimapunt} we need to be more careful in our estimates and to this aim we shall use Lemma \ref{lem:apriori} and Proposition \ref{prop:dercurv} above. Let $\gamma\colon [0,S]\to\Omega$ be a Lipschitz curve: for any $s\in[0,S]$ the bound \eqref{eq:apriorigf} gives (here we are fixing a Borel representative of $u$ and thus of $u_t$, but notice that the estimate \eqref{eq:interm3} does not depend on such choice):
\[
\begin{split}
\sfd^2_\Y&(u_t(\gamma_{0}),u_t(\gamma_{s}))\\
\leq  &e^{-2\lambda t(g(\gamma_{0})\pm |g(\gamma_{0})-g(\gamma_{s})|)}\Big(\sfd_\Y^2(u(\gamma_{0}),u(\gamma_{s}))+2t(g(\gamma_{s})-g(\gamma_{0}))(f(u(\gamma_{0}))-f(u(\gamma_{s})))\\
&+\int_0^{|t(g(\gamma_{0})-g(\gamma_{s}))|}2\Lip^2(f)\theta_\lambda(r)+\lambda^-\big(\sfd_\Y^2({\sf GF}^f_r(u(\gamma_{0})),u(\gamma_{s}))+\sfd_\Y^2({\sf GF}^f_r(u(\gamma_{s})),u(\gamma_{0}))\big)\,\d r\Big),
\end{split}
\]
where the sign in $\pm |g(\gamma_{0})-g(\gamma_{s})|$ depends on the sign of $\lambda$. Now use again the fact that $r\mapsto {\sf GF}^f_r(u(\gamma_{0}))$ is $\Lip(f)$-Lipschitz to get that
\[
\begin{split}
\sfd_\Y^2({\sf GF}^f_r(u(\gamma_{0})),u(\gamma_{s}))&\leq 2r^2\Lip^2(f)+2\sfd_\Y^2(u(\gamma_{0}),u(\gamma_{s})),
\end{split}
\]
notice that the same bounds holds for $\sfd_\Y^2({\sf GF}^f_r(u(\gamma_{s})),u(\gamma_{0}))$, that 
\[
|t(g(\gamma_{0})-g(\gamma_{s}))|\leq ts\Lip(g)\Lip(\gamma)
\]
and that $\theta_\lambda(t)\leq te^{2\lambda^-t}$ to conclude that, for some constant $C$ depending only on $f,g,\Lip(\gamma),T$ and every $t\in[0,T]$, we have
\begin{equation}
\label{eq:interm}
\begin{split}
\sfd^2_\Y&(u_t(\gamma_{0}),u_t(\gamma_{s}))
\leq e^{-2\lambda tg(\gamma_{0})+Cs}\Big(\sfd_\Y^2(u(\gamma_{0}),u(\gamma_{s}))\\
&+2t(g(\gamma_{s})-g(\gamma_{0}))(f(u(\gamma_{0}))-f(u(\gamma_{s})))+Ct^2s^2+Cts\sfd_\Y^2(u(\gamma_{0}),u(\gamma_{s}))\Big).
\end{split}
\end{equation}
Now let $v$ be a test vector field  on $\X$ and $\sfFL_s^v$ its Regular Lagrangian Flow and recall that since $g,f\circ u\in W^{1,2}(\Omega)$, by \eqref{eq:derr} we know that for any $K\subset\Omega$ compact we have
\begin{equation}
\label{eq:interm2}
\frac{g\circ \sfFL_s^v-g}s\to \d g(v)\qquad\text{ and }\qquad\frac{f\circ u\circ \sfFL_s^v-f\circ u}s\to \d (f\circ u)(v)
\end{equation}
in $L^2(K)$ as $s\downarrow0$. Thus writing \eqref{eq:interm} for $\gamma_s:=\sfFL_s^v(x)$ for $\mm$-a.e.\ $x\in\Omega$, dividing by $s^2$, letting $s\downarrow0$ and recalling \eqref{eq:dery} and \eqref{eq:interm2} we conclude that
\begin{equation}
\label{eq:interm3}
|\d u_t(v)|^2\leq e^{-2\lambda tg }\Big(|\d u(v)|^2-2t \,\d g(v)\,\d(f\circ u)(v)+Ct^2\Big)\qquad\mm\text{-}a.e.\ \text{in} \ \Omega,
\end{equation}
having also used the arbitrariness of $K\subset\Omega$ compact and the fact that the Lipschitz constant of $t\mapsto \sfFL_s^v(x)$ is bounded by $\|v\|_{L^\infty}$. We have established \eqref{eq:interm3} for $v$ regular, but both sides of the inequality are continuous w.r.t.\ $L^0$-convergence of uniformly bounded vectors $v$ with values in $L^0_\mm(T\X)$, thus by density we deduce that \eqref{eq:interm3} is valid for any $v\in L^\infty_\mm(T\X)$. Hence writing \eqref{eq:interm3} for  $v$ varying in a local Hilbert base of $L^2_\mm(T\X)$ and adding up we deduce \eqref{eq:stimapunt}. Then \eqref{eq:dervaru} also follows.
\end{proof}
In order to state the analogue of \eqref{eq:casoliscio} in the non-smooth setting we need to recall the notion of measure-valued Laplacian as introduced in \cite{Gigli12} (the presentation that we make here is simplified by the fact that $\RCD$ spaces are infinitesimally Hilbertian). 

Thus let $\X$ be a $\RCD(K,N)$ space, $\Omega\subset\X$ open and bounded and  $f\in W^{1,2}(\Omega)$. We say that $f$ has a measure valued Laplacian in $\Omega$, and write $f  \in D(\DDelta,\Omega)$, provided there is a (signed) Radon measure $\mu$ on $\Omega$ such that
\[ 
\int g\, \d \mu = -\int \la \d f,\d  g\ra \, \d \mm\qquad\forall g \in \Lip_c(\Omega). 
\]
It is clear that this measure is unique and, denoting it by $\DDelta f\restr{\Omega}$, that the assignment $f \mapsto \DDelta f\restr{\Omega}$ is linear. 

We shall need the following criterium for checking whether $f  \in D(\DDelta,\Omega)$:  for $f \in W^{1,2} (\Omega)$ and $h \in L^1(\mm\restr{\Omega})$ we have
\begin{equation}
- \int_\X \la \d f, \d g \ra\, \d \mm \ge \int_\X gh \, \d \mm\quad\forall g \in \Lip_c(\Omega)^+ \quad \Rightarrow \quad f \in D(\DDelta,\Omega)\text{ and } \DDelta f\restr{\Omega}  \ge h\mm.\label{eq:DDeltagenu}
\end{equation}
We are now ready to state and prove the next theorem.
\begin{theorem}\label{thm:nablafu} Let $(\X,\sfd,\mm)$ be a $\RCD(K,N)$ space, $\Y$ be \Cat0 and $\Omega\subset \X$ open and bounded. Also, let   $f \in \Lip(\Y)$ be $\lambda$-convex, $\lambda\in\R$ and $u \in \sfKS^{1,2}(\Omega,\Y)$ be harmonic.

Then, $ f \circ u \in D(\DDelta,\Omega)$ and $\DDelta(f\circ u)\restr{\Omega}$ is a (signed) locally finite Radon measure satisfying
\begin{equation}
\DDelta(f\circ u)\restr{\Omega} \ge \frac{\lambda}{d+2} |\d u|_{\sf HS}^2 \mm. \label{eq:DDeltafu}
\end{equation}
\end{theorem}
\begin{proof} As noticed before Lemma \ref{lem:varEfu}, under the stated assumptions we have $f\circ u\in W^{1,2}(\Omega)$. Now let $g \in \Lip_c(\Omega)^+$ be arbitrary and apply Lemma \ref{lem:varEfu} with these functions $f,g,u$ and define  $u_t \in \sfKS^{1,2}_{\bar{u}}(\Omega,\Y)$ accordingly. Notice that since ${\rm supp}(g)\subset\Omega$, we have that $u_t$ and $u$ agree on a neighbourhood of $\partial\Omega$ and thus have the same boundary value. 

Therefore from the fact that $u$ is harmonic and  \eqref{eq:dervaru} we deduce
\[ 
-\int_\Omega \la \d(f\circ u),\d g \ra \, \d \mm \ge  \frac{\lambda}{d+2}\int_\Omega g|\d u|_{\sf HS}^2\, \d \mm \qquad  \forall g \in \Lip_c(\Omega)^+
\]
and the conclusion comes from   \eqref{eq:DDeltagenu}.
\end{proof}

\begin{corollary}
Let $\Omega\subset \X$ be open, $\Y$ be \Cat0, $\bar{u} \in \sfKS^{1,2}(\Omega,\Y)$, $u$ harmonic map with $\bar{u}$ as boundary values and $f \in \Lip(\Y)$ be $2$-convex. If $f \circ u$  is constant then $u$
itself is constant map.
\end{corollary}

\begin{proof}
Apply Theorem \ref{thm:nablafu}, then $|\d u|_{\sf HS}$ vanishes and conclude.
\end{proof}

Let us now discuss a simple and explicit example of Laplacian of a map. 
\begin{example}\label{ex:s1}{\rm Let $\Y:=\R^2$, $\X:=\R/\Z$ equipped with the standard distances and measure, and $\Omega=\X$. Then a direct application of the definitions in Theorem \ref{thm:defks} show that $u=(u_1,u_2)\colon \X\to\Y$ is in $\sfKS^{1,2}(\X,\Y)$ if and only if $u_1\circ{\rm p},u_2\circ{\rm p}\colon \R\to\R$ are in $W^{1,2}_{loc}(\R)$, where ${\rm p}\colon \R\to\R/\Z=\X$ is the natural projection, with 
\[
\E^\sfKS(u)=\tfrac c2\Big(\int_\X |u_1'|^2(\theta)+ |u_2'|^2(\theta)\,\d\theta \Big),
\]
for some universal constant $c>0$. Then it is clear that $u\in D(\Delta)$ if and only if $ (u_1\circ{\rm p})'',( u_2\circ{\rm p})''\in L^2_{loc}(\R)$ and that in this case
\[
\Delta  u=c( u''_1,u''_2).
\]
Now let $u(\theta):=(\cos(2\pi\theta),\sin(2\pi\theta))$ be the canonical embedding of $\X$ in $\Y$. Then $\Delta u=-u$ and in particular for any $\theta\in\X$ we have that $\Delta u(\theta)\in \T_{u(\theta)}\R^2\sim\R^2$ is orthogonal to the tangent space of $\X$ seen as a subset of $\R^2=\Y$. 

This is interesting because one can define the differential $\d u$ of $u$, even in very abstract situations \cite{GPS18}, by means related to Sobolev calculus on the metric measure space $(\Y,\sfd_\Y,\mu:=u_\sharp(|\d u|_{\sf HS}^2\mm))$ and tangent vector fields in this metric measure space only see directions which are tangent to the graph of $u$ (this is rather obvious in this example, but see for instance \cite{MLP20} for a discussion of this phenomenon in more general cases). This means that, curiously,  $\Delta u$ cannot be computed starting from $\d u$ and using Sobolev calculus in the spirit developed in \cite{Gigli14}, \cite{Gigli17}, simply because $\Delta u$ does not belong to the tangent module $L^2_\mu(T\Y)$
}\fr\end{example}

We conclude pointing out that while in the Definition \ref{def:laplacian} of Laplacian of a map we called into play the space $L^2(u^*\T_G\Y,\mm\restr\Omega)$ as introduced in Definition \ref{def:l2se2}, in some circumstances it might be useful to deal with a notion of Laplacian related to the Borel $\sigma$-algebra $\cB(u^*\T_G\Y)$ - and thus to the characterization given in Proposition \ref{prop:link} -, which however is only available for separable spaces $\Y$. 

In this direction it is worth to underline that one can always reduce to such case thanks to the following two simple results: the first says that given $u\in L^2(\Omega,\Y_{\bar y})$ we can always find a separable \Cat0 subspace $\tilde\Y$ of $\Y$ containing the gradient flow trajectory of $\E^\sfKS_{\bar u}$ starting from $u$, the second ensures that this restriction does not affect the notion of minus-subdifferential. 
\begin{proposition}
Let $(\X,\sfd,\mm)$ be a $\RCD(K,N)$ space, $(\Y,\sfd_\Y,\bar y)$ a pointed \Cat0-space, $\Omega\subset\X$ open, $\bar u\in \sfKS^{1,2}(\Omega,\Y_{\bar{y}})$ and $u\in L^2(\Omega,\Y_{\bar y})$. Also, let $(u_t)$ be the gradient flow trajectory for $\E^\sfKS_{\bar u}$ starting from $u$. 

Then, there exists a separable \Cat0 subspace $\tilde\Y\subset\Y$ such that $\mm(u_t^{-1}(\Y\setminus\tilde\Y))=0$ for every $t\geq 0$. Similarly for the functional $\E^{\sfKS}$.
\end{proposition}
\begin{proof}
From the fact that geodesics on $\Y$ are unique and vary continuously with the endpoint it is easy to see that the closed convex hull of a separable set (i.e.\ the smallest closed and convex set containing the given set) is also separable. Use this and the fact that maps in $L^2(\Omega,\Y_{\bar y})$ are by definition essentially separably valued to find $\tilde\Y\subset\Y$ which is $\Cat0$ with the induced metric and such that $\mm(u_t^{-1}(\Y\setminus\tilde\Y))=0$ for every $t\in\Q^+$. We claim that $\tilde \Y$ satisfies the conclusion. To see this, pick $t\geq 0$, let $(t_n)\subset\Q^+$ be converging to $t$ and up to pass to a non-relabeled subsequence assume that $\sum_n\sfd_{L^2}(u_{t_{n+1}},u_{t_n})<\infty$. Then from the triangle inequality in $L^2(\Omega)$ and the monotone convergence we see that $\|\sum_n\sfd_\Y(u_{t_{n+1}},u_{t_n})\|_{L^2}\leq \sum_n\sfd_{L^2}(u_{t_{n+1}},u_{t_n})<\infty$ so that in particular for $\mm$-a.e.\ $x\in\Omega$ we have $\sum_n\sfd_\Y(u_{t_{n+1}},u_{t_n})(x)<\infty$ which in turn implies that $(u_{t_n}(x))\subset \tilde\Y$ is a Cauchy sequence, so that its limit $v(x)$ also belongs to $\tilde\Y$. The same kind of argument also shows that $(u_{t_n})$ converges to $v$ in $L^2(\Omega,\Y_{\bar y})$ and since we know, by the continuity of $(u_t)$ as $L^2(\Omega,\Y_{\bar y})$-valued curve, that $u_{t_n}\to u_t$ in $L^2(\Omega,\Y_{\bar y})$ we conclude that $u_t=v$, which proves our claim.
\end{proof}
To present our final result we need a bit of notation. Let $\Y$ be a \Cat0-space and $\tilde\Y$ a subspace which is also $\Cat0$ with the induced metric. Call $\mathcal I_{\tilde\Y}^\Y\colon \tilde\Y\to\Y$ the inclusion map. Then for every $y\in\tilde\Y$ the tangent space $\T_y\tilde\Y$ embeds isometrically into $\T_y\Y$ via the continuous extension of the map which sends $\alpha(\G_y^z)'_0\in \T_y\tilde\Y$ to $\alpha(\mathcal I_{\tilde \Y}^\Y(\G_y^z))'_0\in \T_y\Y$. In other words, we can regard a geodesic in $\tilde\Y$ also as a geodesic in $\Y$ and this provides a canonical immersion of $\T_y\tilde\Y$ in $\T_y\Y$ which for trivial reasons is an isometry. Abusing a bit the notation we shall denote such isometry by $\mathcal I_{\tilde\Y}^\Y$.

\begin{proposition}
Let $\Y$ be a \Cat0-space, $\E\colon \Y\to\R\cup\{+\infty\}$ a $\lambda$-convex and lower semicontinuous functional, $(y_t)$ a gradient flow trajectory for $\E$ starting from $y_0\in\Y$ and $\tilde\Y\subset \Y$ a subset which is also a $\Cat0$-space with the induced metric and such that $(y_t)\subset\tilde\Y$. Denote by $\tilde\E$ the restriction of $\E$ to $\tilde\Y$

Then,  $-\partial^-\E(y_0)\neq\emptyset$ if and only if $-\partial^-\tilde\E(y_0)\neq\emptyset$ and letting $v,\tilde v$ be the respective  elements of minimal norm we have $\mathcal I_{\tilde\Y}^\Y(\tilde v)=v$. Moreover, $(y_t)$ is also a gradient flow trajectory for $\tilde E$.
\end{proposition}
\begin{proof}Assume that $-\partial^-\tilde\E(y_0)\neq\emptyset$. Then we know from Theorem \ref{thm:rightD} that $\frac1h(\G_{y_0}^{y_h})'_0\to\tilde v$ as $h\downarrow0$. Then clearly $\mathcal I_{\tilde\Y}^\Y(\frac1h(\G_{y_0}^{y_h})'_0)\to \mathcal I_{\tilde\Y}^\Y(\tilde v)$ and thus by Theorem \ref{thm:rightD} to conclude it is sufficient to prove that $|\partial^-\E|(y_0)<\infty$, because in that case we have that $\mathcal I_{\tilde\Y}^\Y(\frac1h(\G_{y_0}^{y_h})'_0)$ converges to the element of minimal norm in   $-\partial^-\E(y_0)\neq\emptyset$  (which in particular is not empty) as $h\downarrow0$. 

Since $\frac1h(\G_{y_0}^{y_h})'_0\to\tilde v$ we have in particular that $\frac{\sfd_\Y(y_0,y_h)}{h}=|\frac1h(\G_{y_0}^{y_h})'_0|_{y_0}\to |v|_{y_0}$ and thus $S:=\sup_{h\in(0,1)}\frac{\sfd_\Y(y_0,y_h)}{h}<\infty$. By the contractivity property \eqref{eq:contr} we deduce that 
\[
\sup_{t,h\in(0,1)}\frac{\sfd_\Y(y_t,y_{t+h})}{h}<(e^\lambda\vee 1)S=:S'
\]
and thus letting $h\downarrow0$ we deduce that $|\dot y_t^+|\leq S'$ for every $t\in(0,1)$. Taking into account \eqref{eq:metconv} and the lower semicontinuity of the slope recalled in Lemma \ref{lem:slopelsc} we conclude.

Viceversa, assume that $-\partial^-\E(y_0)\neq\emptyset$. Then by Theorem \ref{thm:rightD}  we know that $|\partial^-\E|(y_0)<\infty$ and since trivially we have $|\partial^-\tilde\E|\leq |\partial^-\E|\restr{\tilde \Y}$ we also have $|\partial^-\tilde\E|(y_0)<\infty$. Hence by Theorem \ref{thm:rightD} we deduce $-\partial^-\tilde\E(y_0)\neq\emptyset$ and the first part of the proof applies.

The last statement is a consequence of the first applied to $y_t$ for every $t>0$ and of Corollary \ref{cor:eqfor}.
\end{proof}

\def\cprime{$'$} \def\cprime{$'$}

\end{document}